\documentclass[11pt,twoside,a4paper]{article}
\usepackage
{amsmath,amssymb,mathrsfs}

\usepackage{amsthm}

\usepackage{inputenc}
\usepackage{avant}
\usepackage[english]{babel}
\usepackage{thmtools,thm-restate}
\usepackage[nottoc]
{tocbibind}
\usepackage{graphicx}
\usepackage{enumitem}
\usepackage{tikz}
\usepackage{bbm}
\usepackage{esint}
\usepackage{authblk}
\usepackage{mathtools}
\usepackage{nicematrix}
\NiceMatrixOptions{cell-space-limits = 1pt}

\usepackage[inner=2.4cm,outer=2.4cm,top=2.8cm,bottom=2.8cm]{geometry}
\usepackage[all]{xy}

\usepackage[bookmarks=true]{hyperref}
\usepackage{xcolor}
\hypersetup{
    colorlinks,
    linkcolor={red!50!black},
    citecolor={blue!50!black},
    urlcolor={blue!80!black},
}

\mathtoolsset{showonlyrefs}  

\DeclareMathOperator{\diam}{diam}
\DeclareMathOperator{\supp}{supp}

\newcommand{\norm}[1]{\left\lVert#1\right\rVert}

\newcommand{\normdot}{{\|\!\cdot\!\|}}
\newcommand{\normdotc}{{|\!\cdot\!|}}

\newcommand{\Leb}{\mathscr{L}}

\newcommand{\N}{\mathbb{N}}
\newcommand{\R}{\mathbb{R}}

\newcommand{\p}{\mathtt p} 
\newcommand{\de}{\ensuremath{\, \mathrm d}} 

\newcommand\restr[2]{{
  \left.\kern-\nulldelimiterspace 
  #1 
  \right|_{#2} 
  }}

\newcommand{\weakto}{\rightharpoonup}

\newcommand{\cd}{\mathsf{CD}
}

\newcommand{\RCD}{\mathsf{RCD}}
\newcommand{\MCP}{\mathsf{MCP}}

\newcommand{\X}{\mathsf{X}}
\newcommand{\Y}{\mathsf{Y}}
\newcommand{\Z}{\mathsf{Z}}

\DeclareMathOperator{\res}{restr}

\newcommand{\upden}[1]{\Theta_{#1}^*[\m]}
\newcommand{\loden}[1]{\Theta^{#1}_{*}[\m]}
\newcommand{\di}{\mathsf d} 
\newcommand{\m}{\mathfrak m}
\newcommand{\n}{\mathfrak n}
\newcommand{\s}{\mathfrak s}
\newcommand{\g}{\mathfrak g}

\DeclareMathOperator{\Geo}{Geo}

\DeclareMathOperator{\Tang}{Tang}

\newcommand{\Prob}{\mathscr{P}}
\newcommand{\ProbTwo}{\mathscr{P}_2}

\newcommand{\dis}{\mathcal D}

\newcommand{\sF}{sub-Finsler }
\newcommand{\sr}{sub-Riemannian }

\newcommand{\e}{{\rm e}
}

\newcommand{\hei}{\mathbb{H}}
\newcommand{\Tan}{{\rm Tan}}

\usepackage{titlesec}
\titleformat{\section}
  {\scshape\large\centering}
  {\thesection}{0.5em}{}
\titleformat{\subsection}
  {\scshape\centering}
  {\thesubsection}{0.4em}{}

\title{\textbf{On the rectifiability of $\cd(K,N)$ and $\MCP(K,N)$ spaces with unique tangents}}
\date{\today}

\author{Mattia Magnabosco\footnote{Mathematical Institute, University of Oxford \textit{E-mail}:  \href{mailto:mattia.magnabosco@maths.ox.ac.uk}{mattia.magnabosco@maths.ox.ac.uk}},\quad Andrea Mondino\footnote{Mathematical Institute, University of Oxford \textit{E-mail}:  \href{mailto:andrea.mondino@maths.ox.ac.uk}{andrea.mondino@maths.ox.ac.uk}},\quad Tommaso Rossi\footnote{Laboratoire Jacques-Louis Lions, Sorbonne Universit\'e. \textit{E-mail}: \href{mailto:tommaso.rossi@inria.fr}{tommaso.rossi@inria.fr}}}

\newtheoremstyle{remark}
        {10pt}
        {10pt}
        {}
        {}
        {\itshape}
        {.}
        {.4em}
        {}

\newtheoremstyle{proof}
        {10pt}
        {10pt}
        {}
        {}
        {\itshape}
        {.}
        {.4em}
        {}
        
\newtheoremstyle{definition}
        {10pt}
        {10pt}
        {}
        {}
        {\bfseries}
        {.}
        {.4em}
        {}

\newtheoremstyle{theorem}
        {10pt}
        {10pt}
        {\slshape}
        {}
        {\bfseries}
        {.}
        {.4em}
        {}

\newtheoremstyle{conjecture}
        {10pt}
        {10pt}
        {\slshape}
        {}
        {\itshape}
        {.}
        {.4em}
        {}

\theoremstyle{theorem}
\newtheorem{theorem}{Theorem}[section]
\newtheorem{prop}[theorem]{Proposition}
\newtheorem{corollary}[theorem]{Corollary}
\newtheorem{lemma}[theorem]{Lemma}
\newtheorem{conj}[theorem]{Conjecture}

\theoremstyle{conjecture}
\newtheorem*{conj*}{Conjecture}

\theoremstyle{definition}
\newtheorem{definition}[theorem]{Definition}

\theoremstyle{remark}
\newtheorem{remark}[theorem]{Remark}

\theoremstyle{proof}
\newtheorem*{pro}{Proof}
 {\popQED\end{pro}}

\makeatletter
\renewcommand\xleftrightarrow[2][]{%
  \ext@arrow 9999{\longleftrightarrowfill@}{#1}{#2}}
\newcommand\longleftrightarrowfill@{%
  \arrowfill@\leftarrow\relbar\rightarrow}
\makeatother

\makeindex
\begin{document}

\maketitle

\begin{abstract}
    \noindent We prove rectifiability results for $\cd(K,N)$ and $\MCP(K,N)$ metric measure spaces $(\X,\di,\m)$ with pointwise Ahlfors regular reference measure $\m$ and with $\m$-almost everywhere unique metric tangents. In particular, we show rectifiability if
    \begin{enumerate}
    \item[(i)] $(\X,\di,\m)$ is $\cd(K,N)$ for an arbitrary $N$ and has Hausdorff dimension $n<5$, or
    \item[(ii)] $(\X,\di,\m)$ is $\MCP(K,N)$ and non-collapsed, namely it has Hausdorff dimension $N$. 
    \end{enumerate}
    Our strategy is based on the failure of the $\cd$ condition in \sF Carnot groups, on a new result on the failure of the non-collapsed $\MCP$ on \sF Carnot groups, and on the recent breakthrough by Bate \cite{MR4506771}.
\end{abstract}

\tableofcontents

\section{Introduction}

The synthetic notion of curvature-dimension condition for metric measure spaces, denoted by $\cd(K,N)$, has been introduced in the early 2000s by Lott--Sturm--Villani \cite{sturm2006-1,sturm2006,lott--villani2009}.\ It originates from the observation that, in a Riemannian manifold $(M,g)$, having a uniform lower bound $K\in\R$ on the Ricci curvature and an upper bound $N\in (1,\infty]$ on the dimension can be equivalently characterized as the $(K,N)$-convexity of a suitable entropy functional along geodesics in the Wasserstein space, see \cite{zbMATH01777215, MR2142879, sturm2006} (after \cite{McCannThesis,OttoVillani}). While the definition of Ricci tensor requires the smooth underlying structure of a Riemannian manifold, entropy convexity can be formulated using the language of optimal transport, and relying solely upon a distance and a measure. Therefore, entropy convexity can be formulated in the setting of metric measure spaces and taken as a consistent definition of a curvature-dimension bound. 

Beyond its consistency with the smooth setting, the $\cd(K,N)$ condition has many geometric and analytic consequences. For example, it implies a Bishop--Gromov and a Brunn--Minkowski inequality and it is stable with respect to the (pointed) measured Gromov-Hausdorff convergence \cite{sturm2006, lott--villani2009, Villani-Book-OTO&N}. Moreover, under a suitable non-branching assumption on geodesics, it enjoys the local-to-global property \cite{MR4309491}. The study of $\cd(K,N)$ spaces has led to a deep understanding of the essential features of classical geometric inequalities of Riemannian geometry, allowing for sharpenings of the statements, including novel rigidity and almost-rigidity results \cite{MR1405949,MR3648975,MR4073947}.

A weaker curvature-dimension bound is the \emph{measure contraction property}, or $\MCP(K,N)$ for short, introduced by Ohta in \cite{ohta2007}. This property requires the $(K,N)$-convexity of a suitable entropy functional on the Wasserstein space, along those geodesics that start at a Dirac's mass. Similarly to the $\cd$ condition, $\MCP(K,N)$ is consistent with the classical Riemannian setting, thus the parameters $K$ and $N$ represent a synthetic Ricci curvature lower bound and a dimensional upper bound, respectively. However, while $\cd(K,N)$ implies $\MCP(K,N)$, the latter condition is strictly weaker. For example, there are many \sr manifolds that satisfy the measure contraction property, but are not $\cd(K,N)$ for any choice of the constants $K\in \R$ and $N\in (1,\infty]$, see e.g.\ \cite{MR2520783}.

In this paper, we study the rectifiability of $\cd(K,N)$ and $\MCP(K,N)$ spaces, a long standing open problem in the theory. Given $n\in\N$, a metric measure space $(\X,\di,\m)$ is $(\m,n)$-rectifiable if there exists a subset of full $\m$-measure which can be covered by countably many Lipschitz images of $\R^n$. Due to the broad generality of their definition, $\cd(K,N)$ and $\MCP(K,N)$ spaces may have a very singular geometric structure, and classical analytic tools are not typically available. Consequently, the problem of rectifiability of these spaces is particularly challenging and requires new strategies.
 
With the aim of developing analytic tools in the non-smooth context, Ambrosio, Gigli and Savaré \cite{AmbrosioGigliSavare11} (in the case $N=\infty$) and Gigli \cite{GigliMemoirs} (in the case $N<\infty$) considered the class of $\RCD(K,N)$ metric measure spaces, namely $\cd(K,N)$ spaces satisfying an additional assumption called infinitesimal Hilbertianity. The latter asks the Sobolev space $W^{1,2}(\X,\di,\m)$ (defined according to Cheeger \cite{cheeger}) to be Hilbert or, equivalently,  to having a linear heat flow. In this regard,  infinitesimal Hilbertianity forces the space to be Riemannian-like \cite{MR2917125}.
 
The rectifiability of $\RCD(K,N)$ spaces was proved by the second-named author and Naber in \cite{MR3945743}. In particular, they showed that any $\RCD(K,N)$ space $(\X,\di,\m)$ can be decomposed, up to an $\m$-null set, in $\left\lfloor N \right\rfloor $ ``regular'' sets, denoted by  $\mathscr{R}_k$, which are $k$-rectifiable. This result was then independently refined in \cite{MR3801291,MR4367429,MR3701738}, where it was proved that, for every $k\leq N$, the measure $\m_{|\mathscr{R}_k}$ is absolutely continuous with respect to the Hausdorff measure $\mathcal H^k$. Finally, Brué and Semola \cite{MR4156601} were able to prove the constancy of dimension for $\RCD(K,N)$ spaces, namely that there exists exactly one regular set $\mathscr{R}_k$ having positive measure. Let us mention that both the rectifiability and the constancy of the dimension had been previously obtained in the framework of Ricci-limits by Cheeger-Colding \cite{CC-JDG-I,CC-JDG-II, CC-JDG-III} and by Colding-Naber \cite{CN-12}, respectively. A major difference between establishing the results for Ricci-limits and for $\RCD(K,N)$ spaces is that, in the former, it is possible to appeal on the existence of approximations by smooth manifolds with Ricci curvature bounded below, while in the latter one has to work directly on the non-smooth metric measure space satisfying the Ricci curvature lower bounds in a synthetic sense. 

The aforementioned structural results for $\RCD(K,N)$ spaces heavily rely on the infinitesimal Hilbertianity assumption. For example, the very first step of the construction of the regular sets in \cite{MR3945743} uses the splitting theorem which fails for a general $\cd(K,N)$ (or $\MCP(K,N)$) space. In this paper, we propose a different approach to the problem of rectifiability, based on a recent breakthrough by Bate \cite{MR4506771}. 

Our first main result (Theorem \ref{thm:rectifiability<5}) deals with the rectifiability of $\cd(K,N)$ spaces having a small (pointwise) Ahlfors dimension. The latter is a real number $n>0$ for which the $n$-th upper and lower $\m$-densities $\upden{n}$ and $\loden{n}$ are both positive and finite $\m$-almost everywhere. We refer to Section \ref{sec:meta_rectifiability} for the precise definitions. 

\begin{theorem}\label{thm:main}
    Given $K\in\R$ and $N\in (1,\infty)$, let $(\X,\di,\m)$ be a $\cd(K,N)$ space. Assume that:  
    \begin{itemize}
        \item[i)] for $\m$-a.e.\ $x\in\X$, $\Tan_{{\rm pGH}}(\X, \di,x)$ contains a single element, up to isometry;
        \item[ii)] there exists $n\in (0,5)$ such that $0<\loden{n}(x)\leq \upden{n}(x)<\infty$, for $\m$-a.e.\ $x\in\X$. 
    \end{itemize}
     Then:
     \begin{itemize}
\item     $n$ is an integer;
\item $(\X,\di,\m)$ is $(\m,n)$-rectifiable;
\item for $\m$-a.e.\ $x\in\X$, $\Tan_{{\rm pmGH}}(\X, \di,\m,x)$ contains a single element (up to isomorphism) and its metric part is isometric to an $n$-dimensional Banach space.
\end{itemize}
\end{theorem}

Let us comment on the above statement. The two key assumptions at the core of our strategy are the $\m$-a.e.\;uniqueness of the Gromov-Hausdorff metric tangent, and the pointwise $n$-Ahlfors regularity of the reference measure.

On the one hand, the uniqueness of metric tangents may appear as a restrictive assumption; however, let us stress that it \emph{does not} rule out unrectifiable metric spaces, such as \sF Carnot groups.\ We remark that, to the best of our knowledge, there are no examples of $\cd(K,N)$ spaces having non-unique Gromov-Hausdorff metric tangents on a set of positive $\m$-measure. 

On the other hand, the Ahlfors regularity of the measure is not implied by the $\cd(K,N)$ condition, as shown by \cite{MR4402722,breaking}. This is in sharp contrast with $\RCD$ spaces, where it is a consequence of the constancy of dimension.

We conclude by observing that Theorem \ref{thm:main} can be adapted to accommodate for a stratification of $\X$, according its Ahlfors dimensions. However, differently from $\RCD$ spaces, where a stratification can be defined using the splitting theorem, cf.\ \cite{MR3945743}, in the generality of $\cd(K,N)$ spaces, an operative way to stratify $\X$ is less clear. See Remark \ref{rmk:stratification} for further details. 

\medskip 

Our second main result (Theorem \ref{thm:rectifiability_MCP}) instead gives a rectifiability result for ``non-collapsed'' $\MCP(K,N)$ metric measure spaces, without any restriction on the Ahlfors dimension. 

\begin{theorem}
\label{thm:main2}
    Given $K\in\R$ and $N\in (1,\infty)$, let $(\X,\di,\m)$ be an $\MCP(K,N)$ space. Assume that:  
    \begin{itemize}
        \item[i)] for $\m$-a.e.\ $x\in\X$, $\Tan_{{\rm pGH}}(\X, \di,x)$ contains a single element, up to isometry;
        \item[ii)] for $\m$-a.e.\ $x\in\X$, $0<\loden{N}(x)\leq \upden{N}(x)<\infty$. 
    \end{itemize}
    Then:
    \begin{itemize}
    \item  $N$ is an integer;
    \item $(\X,\di,\m)$ is $(\m,N)$-rectifiable;
    \item for $\m$-a.e.\ $x\in\X$, $\Tan_{{\rm pmGH}}(\X, \di,\m,x)$ contains a single element (up to isomorphism) and its metric part is isometric to a $N$-dimensional Banach space.
    \end{itemize}
\end{theorem}

In the above statement, the non-collapsing condition of $(\X,\di,\m)$ corresponds to assumption {\slshape ii)}, stating that the Ahlfors dimension coincides with $N$, which in turn plays the role of a synthetic upper bound on the dimension of $X$ in the $\MCP(K,N)$ condition.

What makes Theorem \ref{thm:main2} particularly interesting is that there are plenty of examples of $\MCP(K,N)$ spaces with a.e.\;unique metric tangents, which are not rectifiable (such as the \sr Heisenberg group). Thus, remarkably, the non-collapsing assumption {\slshape ii)} in Theorem \ref{thm:main2} is enough to rule out these examples and prove rectifiability. See Remark \ref{Rem:AfterThm1.2} for more details.

\paragraph{Strategy of the proofs.} The strategy for the proof of Theorem \ref{thm:main} is based on the following result proven in \cite{magnabosco2023failure,borza2024measure} about \sF Carnot groups of low Hausdorff dimension. Recall that a sub-Finsler Carnot group is a stratified, nilpotent Lie group equipped with a left-invariant distance. In the following, we denote the Hausdorff measure in dimension $n$  by $\mathcal H^n$.

\begin{theorem}[{\cite{magnabosco2023failure,borza2024measure}}]
\label{thm:aux_intro}
    Let $(G,\di_{SF})$ be a \sF Carnot group of Hausdorff dimension $n\in(0,5)$. If $(G,\di_{SF},\mathcal H^n)$ satisfies the $\cd(0,N)$ condition for some $N > 1$, then $G$ is a finite-dimensional Banach space.
\end{theorem}
     
\noindent Then, we proceed as follows: by the assumption on the uniqueness of Gromov-Hausdorff metric tangents, a result by Le Donne \cite{MR2865538} ensures that at $\m$-a.e.\;$x\in \X$ the unique tangent space is a \sF Carnot group. Additionally, the pointwise $n$-Ahlfors regularity of $\m$ guarantees that all tangent measures are uniformly $n$-Ahlfors regular as well (see Proposition \ref{prop:tan=tan}). Consequently, at $\m$-a.e.\ $x\in\X$, we find that a metric measure tangent space of $(\X,\di,\m)$ is a \sF Carnot group of Hausdorff dimension $n\in (0,5)$, equipped with $\mathcal H^n$. Invoking the stability of the $\cd(K,N)$ condition under (pointed) measured Gromov-Hausdorff convergence, we find that this \sF Carnot group (equipped with the Hausdorff measure) is also a $\cd(0,N)$ space. Thus,  Theorem \ref{thm:aux_intro} forces such a \sF Carnot group to be a finite-dimensional Banach space. Finally, the rectifiability of $\X$ follows by the characterization of Bate  \cite{MR4506771}, stating that a subset $E\subset\X$ is $n$-rectifiable if and only if the (pointed) Gromov-Hausdorff metric measure tangents are finite-dimensional Banach spaces at almost every point of $E$.

\medskip

The aforementioned strategy is quite flexible: on the one hand, it can be used to prove rectifiability for any Ahlfors dimension (see Theorem \ref{thm:meta_rectifiability}), provided that the following conjecture on the validity of $\cd$ condition in \sF Carnot groups is verified. 

\begin{conj*}[Conjecture \ref{conj:carnot_nocd}]
     A \sF Carnot group which, equipped with the Hausdorff measure, satisfies the $\cd(0,N)$ condition for some $N > 1$, is a finite-dimensional Banach space.
\end{conj*}

On the other hand, the very same strategy is used to prove Theorem \ref{thm:main2}. In this case, the non-collapsing assumption guarantees that, at almost every point, the unique Gromov-Hausdorff metric tangent is a \sF Carnot group of Hausodorff dimension $N$ and satisfying the $\MCP(0,N)$ property. Then, we conclude rectifiability, exploiting the following result of independent interest, see Theorem \ref{thm:noMCP}. 

\begin{theorem}
    \label{thm:main3}
    Let $(G, \di_{SF})$ be a sub-Finsler Carnot group with homogeneous dimension $N$. Then, either $G$ is commutative or $(G, \di_{SF}, \mathcal{H}^N)$ does not satisfy the $\MCP(0,N)$ condition.    
\end{theorem}

Theorem \ref{thm:main3} shows that (non-commutative) \sF Carnot groups cannot be non-collapsed $\MCP(0,N)$ spaces. An equivalent formulation is that, for a  \sF Carnot group $(G,\di_{SF},\mathcal H^N)$, it must hold
\begin{equation}
    N_{\text{curv}}=N\qquad\Longleftrightarrow\qquad G \text{ is commutative,}
\end{equation}
where the curvature exponent $N_{\text{curv}}$ is the optimal dimensional parameter for which $\MCP$ holds. In this regard, Theorem \ref{thm:main3} extends known results for \sr Carnot groups, where the curvature exponent is always greater than or equal to the Hausdorff dimension plus one, see \cite{MR3852258}.

The proof of Theorem \ref{thm:main3} relies on a careful blow-up procedure, inspired by \cite{EeroELD}. Firstly, if the non-collapsed $\MCP$ holds, given an optimal geodesic plan connecting the Dirac's mass centered at the origin $\delta_\e$ and the uniform measure on the unit ball $\mu$, its blow-up is again a geodesic plan with the same marginals. Secondly, up to iterating the blow-up procedure, we show that there exists an optimal geodesic plan connecting $\delta_\e$ and $\mu$, which is concentrated on a set of geodesics that project to geodesics of the abelianisation of $G$.  However, the latter set contains geodesics that, when evaluated at time $t\in (0,1)$, cover the ball of radius $t$ if and only if $G$ is commutative.

Note that there are examples of \sF Carnot groups where the set of $t$-midpoints $M_t(\e, \bar B(\e,1))$ between the identity element $\e$ and the unit ball $\bar B(\e,1)$ coincides with the rescaled ball for every $t\in [0,1]$. Thus,
        \begin{equation}
             \mathcal H^N(M_t(\e,\bar B(\e,1)))=\mathcal H^N(\bar B(\e,t) )=t^N\mathcal H^N(\bar B(\e,1)),\qquad\forall\,t\in [0,1].
        \end{equation}
For this reason, in the proof of Theorem \ref{thm:main3},  we contradict directly the definition of $\MCP(0,N)$, rather than its standard implication on the contraction rate of the set of $t$-midpoints. See Remark \ref{rmk:final_rmk} for details.

\paragraph{Structure of the paper.} In Section \ref{sec:prelim}, we give the necessary preliminaries on $\cd$ and $\MCP$ spaces and \sF Carnot groups. In Section \ref{sec:meta_rectifiability}, we detail the general strategy for rectifiability and prove Theorem \ref{thm:main}. Moreover, we show how to deduce the rectifiability if Conjecture \ref{conj:carnot_nocd} is verified; see Theorem \ref{thm:meta_rectifiability}. Section \ref{sec:rectifiability_non-collapsed_mcp} is devoted to the proof of Theorem \ref{thm:main3}, from which we deduce Theorem \ref{thm:main2}.

\subsection*{Acknowledgments}  
The authors thank Camillo Brena and Enrico Pasqualetto for discussions on metric rectifiability.

M.\,M.\;ac\-knowl\-edges support from the Royal Society through the Newton International Fellowship (award number:\ NIF$\backslash$R1$\backslash$231659). A.\,M.\;acknowledges support from the European Research Council (ERC) under the European Union's Horizon 2020 research and innovation programme, grant agreement No.\;802689 ``CURVATURE''.  T.\,R.\;acknowledges support from the ANR-DFG project ``CoRoMo'' (ANR-22-CE92-0077-01).

Part of this research was carried out while M.\,M.\;and T.\,R.\;were hosted at the Oberwolfach Research Institute for Mathematics as Oberwolfach Research Fellows. They thank the MFO for providing a stimulating and fruitful research environment.\ Part of this research was carried out at the Hausdorff Institute of Mathematics in Bonn, during the trimester program  ``Metric Analysis''. The authors wish to express their appreciation to the institution for the stimulating atmosphere, and they ackonowledge support  by the Deutsche Forschungsgemeinschaft (DFG, German Research Foundation) under Germany's Excellence Strategy – EXC-2047/1 – 390685813.

For the purpose of Open Access, the authors have applied a CC BY public copyright licence to any Author Accepted Manuscript (AAM) version arising from this submission.

\section{Preliminaries}
\label{sec:prelim}

\subsection{Synthetic Curvature-Dimension bounds}

A metric measure space is a triple $(\X,\di,\m)$ where $(\X,\di)$ is a complete and separable metric space and $\m$ is a non-null Borel measure on it, finite on compact sets. In the following, $C([0, 1], \X)$ is the space of continuous curves from $[0, 1]$ to $\X$. A curve $\gamma\in C([0, 1], \X)$ is said to be \emph{a geodesic} if 
\begin{equation}
    \di(\gamma(s), \gamma(t)) = |t-s| \cdot  \di(\gamma(0), \gamma(1)) \qquad \forall s,t\in[0,1].
\end{equation}
We denote by $\Geo(\X)$ the set of geodesics of $(\X,\di)$. The metric space $(\X,\di)$ is said to be \emph{geodesic} if every pair of points $x,y \in \X$ can be connected with a curve $\gamma\in \Geo(\X)$. 
For any $t \in [0, 1]$, we define the evaluation map $e_t \colon C([0, 1], \X) \to \X$ by setting $e_t(\gamma) := \gamma(t)$.

We denote by $\Prob(\X)$ the set of Borel probability measures on $\X$ and by $\Prob_2(\X) \subset \Prob(\X)$ the set of those having finite second moment. We endow $\Prob_2(\X)$ with the Wasserstein distance $W_2$: 
\begin{equation}
\label{eq:defW2}
    W_2(\mu_0, \mu_1) := \inf_{\pi \in \mathsf{Adm}(\mu_0,\mu_1)}  \left( \int \di^2(x, y) \, \de \pi(x, y)\right)^{1/2},
\end{equation}
where $\mathsf{Adm}(\mu_0, \mu_1)$ is the set of all the admissible transport plans between $\mu_0$ and $\mu_1$, namely all the measures in $\Prob(\X\times\X)$ such that $(\p_1)_\# \pi = \mu_0$ and $(\p_2)_\#\pi = \mu_1$. We denoted by $\p_i:X\times X \to X$, $i=1,2$, the projection on the $i^{th}$ coordinate, and by $(\p_i)_\#: \Prob(\X\times\X)\to  \Prob(\X) $ the associated push-forward map.   The infimum in \eqref{eq:defW2} is always attained and the measures in $\mathsf{Adm}(\mu_0, \mu_1)$ realising the minimum are called optimal transport plans. The metric space $(\Prob_2(\X),W_2)$ is itself complete and separable, moreover, if $(\X,\di)$ is geodesic, then $(\Prob_2(\X),W_2)$ is geodesic as well. Furthermore, every geodesic $(\mu_t)_{t\in [0,1]}$ in $(\Prob_2(\X),W_2)$ can be represented with a measure $\eta \in \Prob(\Geo(\X))$, meaning that $\mu_t = (e_t)_\# \eta$. We refer the reader to \cite{MR4294651} for more details.

With these tools, one can define the $\cd(K,N)$ condition, pioneered by Sturm and Lott--Villani \cite{sturm2006-1,sturm2006,lott--villani2009}, and the measure contraction property $\MCP(K,N)$, introduced by Ohta in \cite{ohta2007}. These conditions aim to generalize, to the context of metric measure spaces, the notion of having Ricci curvature bounded from below by $K\in\R$ and dimension bounded above by $N \in (1,\infty)$. Both consist in requiring a suitable $(K,N)$-convexity property of the so-called R\'enyi entropy functionals.

The distortion coefficients are defined as follows: for every $K \in \R$ and $N\in (1,\infty)$, set
\begin{equation}\label{eq:tau}
    \tau_{K,N}^{(t)}(\theta):=t^{\frac{1}{N}}\left[\sigma_{K, N-1}^{(t)}(\theta)\right]^{1-\frac{1}{N}},\qquad \forall\,t\in[0,1],\  \forall\,\theta\in[0,+\infty),
\end{equation}
where
\begin{equation}
\sigma_{K,N}^{(t)}(\theta):= 
\begin{cases}

\displaystyle  \frac{\sin(t\theta\sqrt{K/N})}{\sin(\theta\sqrt{K/N})} & \textrm{if}\  N\pi^{2} > K\theta^{2} >  0, \crcr
t & \textrm{if}\ 
K =0,  \crcr
\displaystyle   \frac{\sinh(t\theta\sqrt{-K/N})}{\sinh(\theta\sqrt{-K/N})} & \textrm{if}\ K < 0.
\end{cases}
\end{equation}

\begin{definition}\label{def:CD}
Given $K\in \R$ and $N\in (1,\infty)$, a metric measure space $(\X,\di,\m)$ is said to be a $\cd(K,N)$ space (or to satisfy the $\cd(K,N)$ condition) if for every pair of measures $\mu_0=\rho_0\m,\mu_1= \rho_1 \m \in \Prob_2(\X)$, absolutely continuous with respect to $\m$, there exists a $W_2$-geodesic $(\mu_t)_{t\in [0,1]}$ connecting them and induced by $\eta \in \Prob(\Geo(\X))$, such that for every $t\in [0,1]$, $\mu_t =\rho_t \m \ll \m$ and the following inequality holds for every $N'\geq N$ and every $t \in [0,1]$:
\begin{equation}\label{eq:CDcond}
    \int_\X \rho_t^{1-\frac 1{N'}} \de \m \geq \int_{\X \times \X} \Big[ \tau^{(1-t)}_{K,N'} \big(\di(x,y) \big) \rho_{0}(x)^{-\frac{1}{N'}} +    \tau^{(t)}_{K,N'} \big(\di(x,y) \big) \rho_{1}(y)^{-\frac{1}{N'}} \Big]    \de\pi( x,y),
\end{equation}
where $\pi= (e_0,e_1)_\# \eta$.
\end{definition}

 The heuristic idea behind the measure contraction property $\MCP(K,N)$ is to require the $\cd(K,N)$ condition to hold when the first marginal degenerates to a delta-measure $\delta_x$, and the second marginal is $\frac{\m|_A}{\m(A)}$, for some Borel set $A\subset\X$ with $0<\m(A)<\infty$.

\begin{definition}
\label{def:mcp}
    Given $K\in\R$ and $N\in (1,\infty)$, a metric measure space $(\X,\di,\m)$ is said to satisfy the measure contraction property $\mathsf{MCP}(K,N)$ if for every $x\in\text{supp}(\m)$ and every Borel set $A\subset\X$ with $0<\m(A)<\infty$, there exists a Wasserstein geodesic, induced by some $\eta \in \Prob(\Geo(\X))$,  connecting $\delta_x$ and $\frac{\m|_A}{\m(A)}$ such that, for every $t\in[0,1]$,
    \begin{equation}
    \label{eq:mcp_def}
        \frac{1}{\m(A)}\m\geq(e_t)_\#\Big(\tau_{K,N}^{(t)}\big(\di(\gamma(0),\gamma(1))\big)^N\eta(\text{d}\gamma)\Big).
    \end{equation}
\end{definition}

 For every $K\in \R$ and every $N\in (1,\infty)$, the $\cd(K,N)$ condition implies $\MCP(K,N)$ and the latter is strictly weaker. Furthermore, we mention here that $\MCP(K,N)$ (and thus $\cd(K,N)$) implies a sharp Bishop--Gromov inequality on $(\X,\di,\m)$, which in turn implies that $\m$ is locally uniformly doubling on its support, i.e.\ for every $R>0$, there exists $C=C(R)>0$ such that 
    \begin{equation}
    \label{eq:loc_unif_doubling}
        \m(B(x,2r))\leq C\m(B(x,r)),\qquad\forall\,x\in \supp(\m),\ \forall\,r\in[0,R].
    \end{equation}

\begin{remark}
    If $(\X,\di,\m)$ is a metric measure space satisfying $\MCP(K,N)$ (or $\cd(K,N)$) for some $K\in\R$ and $N\in (1,\infty)$, then, $(\X,\di)$ is a proper, geodesic metric space, see for instance \cite{sturm2006}.
\end{remark}

\subsection{Sub-Finsler Carnot groups}\label{subsection:carnot}

We recall some basic facts on \sF Carnot groups. For a complete account, we refer the reader to the monograph \cite{libroELD}. 

Let $(G,\cdot)$ be a simply connected Lie group, with identity element $\e$. We say that $G$ is a Carnot group if its Lie algebra $\g := T_\e G$ admits a stratification 
\begin{equation}
\label{eq:stratification}
    \g = V_1 \oplus\ldots\oplus V_\s, 
\end{equation} 
i.e., $V_1, \ldots , V_\s\subset\g$ are subspaces in direct sum such that $V_{j+1} = [V_j , V_1]$, with $V_{\s+1} := \{0\}$. The maximal number $\s\in\N$ for which $V_\s\neq\{0\}$  is called the \emph{step} of the Carnot group, while the rank is defined as $k:=\dim(V_1)$. Note that $\exp(V_\s)\subset Z(G)$, 
where $Z(G)$ is the centre of the group and $\exp:\g\to G$ denotes the exponential map of $G$. In addition, by definition, the first layer $V_1$ is bracket generating  and induces a left-invariant distribution $\dis\subset TG$:
\begin{equation}
\label{eq:left-inv_distribution}
    \dis_g := (L_g)_*V_1,\qquad\forall\,g \in G,
\end{equation}
where $L_g:G\to G$ is the multiplication on the left by the element $g\in G$. Fix a norm $\normdot$ on the vector space $V_1$. By left-translation, the norm on $V_1$ induces a norm on $\dis$. Indeed, if $V\in\dis_g$, then $V=(L_g)_*v$ for some $v\in V_1$, and then
\begin{equation}
    \|V\|_g:= \norm{(L_g)^{-1}_*V}=\norm{v}. 
\end{equation}
The associated \sF distance is defined by length-minimisation among admissible curves. More precisely, we say that $\gamma:[0,1]\to G$ is admissible (or horizontal) if it is absolutely continuous in coordinates and, for a.e.\ $t\in[0,1]$, $\dot\gamma(t)\in \dis_{\gamma(t)}$. Given an admissible curve, its length is 
\begin{equation}
    \ell(\gamma):=\int_0^1\norm{\dot\gamma(t)}_{\gamma(t)}\de t.
\end{equation}
Then, for every $p,q\in G$, the \sF distance between $p$ and $q$ is
\begin{equation}
    \di_{SF}(p, q) := \inf\{ \ell(\gamma) : \gamma:[0,1]\to G,\text{ admissible and such that } \gamma(0)=p, \gamma(1)=q\}.
\end{equation}
The metric space $(G,\di_{SF})$ is a \sF Carnot group and we refer to $\normdot$ as the reference norm of the \sF Carnot group.

\begin{remark}
\label{rmk:hom_dim}
    The topological dimension of $G$ is given by $n:=\sum_{i=0}^\s \dim(V_i)$, while its homogeneous dimension is $\mathcal Q:=\sum_{i=0}^\s i\dim(V_i)$. Note that the homogeneous dimension of a \sF Carnot group coincides with its Hausdorff dimension.
\end{remark}

\paragraph{Dilations on Carnot groups.}

Let $(G,\di_{SF})$ be a \sF Carnot group, with Lie algebra $\g$. Let $\tilde\delta_\lambda : \g \to \g$ be the dilation of factor $\lambda\in\R$ associated with the stratification \eqref{eq:stratification}, i.e.\ 
\begin{equation}
    \tilde\delta_{\lambda}v=\lambda^iv,\qquad\forall\,v\in V_i,\,i=1,\ldots,\s.
\end{equation} 
Then, the (anisotropic) dilation $\delta_\lambda : G \to G$ of the group of factor $\lambda$ is the only group automorphism such that $(\delta_\lambda)_* = \tilde\delta_\lambda$. By definition, for every $\lambda,\eta\in\R$, we have $\delta_\lambda \circ \exp = \exp\circ\,\tilde\delta_\lambda$ and $\delta_\lambda \circ\,\delta_\eta= \delta_{\lambda\eta}$. In addition, 
\begin{equation}\label{eq:commutation_product-dilation}
    \delta_\lambda (x \cdot y) = \delta_\lambda(x) \cdot \delta_\lambda(y), \qquad\forall\,x,y\in G.
\end{equation}
Finally, the \sF distance is $1$-homogeneous under the action of dilations.

\begin{prop}
    Let $(G,\di_{SF})$ be a \sF Carnot group and let $\lambda\in\R$. Then,
    \begin{equation}
        \di_{SF}(\delta_\lambda(x),\delta_\lambda(y)) = |\lambda|\,\di_{SF}(x,y), \qquad\forall\,x, y \in G,\,\lambda \in \R.       
    \end{equation}
    In particular, $\delta_{-1}:G\to G$ is a non-trivial isometry of $G$.
 \end{prop}
 
\paragraph{The Heisenberg group \texorpdfstring{$\hei$}{H}.}

We recall here the first important example of a \sF Carnot group, namely the \sF Heisenberg group. Consider the Lie group $G=\R^3$, equipped with the non-commutative group law, defined by
\begin{equation}
    (x, y, z) \cdot (x', y', z') = \bigg(x+x',y+y',z+z'+\frac12(xy' - x'y)\bigg),\qquad\forall\,(x, y, z), (x', y', z')\in\R^3,
\end{equation}
with identity element $\e=(0,0,0)$. Define the left-invariant vector fields
\begin{equation}
    X:=\partial_x-\frac{y}2\partial_z,\qquad Y:=\partial_y+\frac{x}2\partial_z.
\end{equation}
Then, $[X,Y]=\partial_z=:Z$ and the Lie algebra of $G$ admits the stratification: 
\begin{equation}
    \g=V_1\oplus V_2,
\end{equation}
where $V_1=\text{span}\{X,Y\}$ and $V_2:=\text{span}\{Z\}$. The resulting Carnot group is the Heisenberg group $\hei$. If we equip $V_1$ with a norm $\normdot$, the metric space $(\hei,\di_{SF})$ is a \sF Heisenberg group. The homogeneous dimension of a \sF Heisenberg group is $\mathcal Q=4$, and thus its Hausdorff dimension is also equal to $4$, cf.\ Remark \ref{rmk:hom_dim}. Note that, the Heisenberg group is the unique non-commutative Carnot group with Hausdorff dimension equal to $4$ (see Theorem \ref{thm:classification} below). 

The curvature-dimension condition $\cd(K,N)$ is known to fail in every \sF Heisenberg group, according to \cite{borza2024measure} (see also \cite{magnabosco2023failure}). 

\begin{theorem}[{\cite[Thm.\ 1.6]{borza2024measure}}]
\label{thm:nocd_hei}
    Let $(\hei,\di_{SF})$ be a \sF Heisenberg group, equipped with a norm $\normdot$ and with a positive smooth measure $\m$. Then, the metric measure space $(\hei,\di_{SF},\m)$ does not satisfy the $\cd(K,N)$ condition, for every $K\in\R$ and $N\in (1,\infty)$.
\end{theorem}

Let us mention that, nonetheless, the \emph{sub-Riemannian} Heisenberg group satisfies a modified curvature-dimension condition with different distortion coefficients, see \cite{MR4019096}; see also \cite{BaMoRi-Unified} for a unified synthetic setting of generalized curvature-dimension conditions for Riemannian and sub-Riemannian structures.

\paragraph{Classification of low-dimensional Carnot groups.} 

We report here a useful result, regarding the classification of low-dimensional Carnot groups, see \cite[Thm.\ 10.81]{ABB-srgeom} or \cite[Sec.\ 10.1.3]{libroELD}. Recall that, given two Lie groups $G,H$ a Lie group isomorphism between them is a diffeomorphism $\phi : G \to H$ that is also a group homomorphism.

\begin{theorem}
\label{thm:classification}
    Let $(G,\di_{SF})$ be a Carnot group of homogeneous dimension $\mathcal Q\leq 4$. Then, $G$ is isomorphic as a Lie group to one of the following:
    \begin{itemize}
        \item[i)] $\R^n$, with $n=1,\ldots,4$, if $G$ is commutative;
        \item[ii)] $\hei$, if $G$ is not commutative and $\mathcal Q=4$.
    \end{itemize}
\end{theorem}

 By endowing the distribution with a norm, Theorem \ref{thm:classification} implies the following result. 

\begin{corollary}
    \label{cor:classification}
    Let $(G,\di_{SF})$ be a \sF Carnot group of homogeneous dimension $\mathcal Q\leq 4$. Then, $G$ is isometric to one of the following:
    \begin{itemize}
        \item[i)] $(\R^n,\normdotc)$, with $n=1,\ldots,4$ and $\normdotc$ a suitable norm, if $G$ is commutative;
        \item[ii)] a \sF Heisenberg group $(\hei,\di_{SF})$, if $G$ is not commutative and $\mathcal Q=4$.
    \end{itemize}
\end{corollary}

\section{Rectifiability of \texorpdfstring{$\cd(K,N)$}{CD(K,N)} spaces with unique tangents}
\label{sec:meta_rectifiability}

The goal of this section is to state and prove the main results on the rectifiability of $\cd(K,N)$ spaces with a.e.\;unique metric tangents and controlled dimension. In the process, we develop a general strategy that, in the last section, will allow to prove rectifiability results also for non-collapsed $\MCP(K,N)$ spaces with a.e.\;unique tangents.

We recall two different notions of $n$-rectifiability: the first one is the metric version of the classical definition in Euclidean spaces and stems from \cite{AmbKirch-MathAnn}, while the second one is tailored for metric measure spaces. In Remark \ref{rmk:rectifiability} below, we explain the advantage of working with both. Here and below, $\mathcal H^n$ denotes the $n$-dimensional Hasurdorff measure. 

\begin{definition}\label{def:rect1}
    Fix $n\in\N$ and let $(\X,\di)$ be a metric space. We say that a Borel subset $E\subset \X$ is $n$-rectifiable if there exists a countable collection of Lipschitz maps $\{f_i:B_i\to \X\}_{i\in \N}$, where $B_i \subset \R^n$ is Borel, such that 
    \begin{equation*}
        \mathcal H^n\left( E \setminus \bigcup_{i\in \N} f_i(B_i) \right)= 0.
    \end{equation*}
\end{definition}
 Observe that if a set $E\subset X$ is $n$-rectifiable according to Definition \ref{def:rect1}, then it is $n'$-rectifiable for any $n'\geq n$. To avoid this feature, one may add the further requirement that $\mathcal H^n\left(\bigcup_{i\in \N} f_i(B_i)\right)>0$. Anyway, we decided to stick to the definition that appears in \cite{AmbKirch-MathAnn, MR4506771}, for the sake of consistency in the literature. Instead, we implement this observation in the next definition of  $(\m,n)$-rectifiability for a metric measure space, since it does not seem to create a conflict of notation with previous literature.

\begin{definition}\label{def:rect2}
    Fix $n\in\N$ and let $(\X,\di,\m)$ be a metric measure space. We say that $\X$ is $(\m,n)$-rectifiable if there exists a countable collection of Lipschitz maps $\{f_i:B_i\to \X\}_{i\in \N}$, where $B_i \subset \R^n $ is Borel, such that 
    \begin{equation}
    \label{eq:m,n-rectifiability}
         \mathcal H^n\left(\bigcup_{i\in \N} f_i(B_i)\right)>0\qquad\text{and}\qquad
        \m\left( \X \setminus \bigcup_{i\in \N} f_i(B_i) \right)= 0.
    \end{equation}
\end{definition}

\begin{remark}[On the definitions of rectifiability]
\label{rmk:rectifiability}
    Definition \ref{def:rect1} is the one adopted in \cite{MR4506771}. As our strategy relies on applying the main result of \cite{MR4506771}, we will work with this one in the proofs below. Instead, the rectifiability notion of Definition \ref{def:rect2} is the one we are going to prove in our main theorems. The reason why this second notion is better suited to the context of metric measure spaces (especially the ones satisfying a curvature-dimension bound) comes from an example presented in \cite{MR4566173,MR4431126}. In fact, there is an $\RCD(K,N)$ space $(\X,\di,\m)$ which is $(\m,n)$-rectifiable for some $n$, but has Hausdorff dimension strictly larger than $n$. In particular, $\X$ is not $n$-rectifiable in the sense of Definition \ref{def:rect1}, but it contains an $\m$-full measure subset which is $n$-rectifiable. In conclusion, in the setting of metric measure spaces $(\X,\di,\m)$, we take advantage of the reference measure $\m$, to use a more flexible definition. 
    In addition, observe that the first requirement of \eqref{eq:m,n-rectifiability} ensures that $n$ is the ``optimal'' rectifiable dimension, in the sense that if $(\X,\di,\m)$ is $(\m,n)$-rectifiable, it is not $(\m,n')$-rectifiable for $n'>n$.
\end{remark}

\subsection{On the notion of convergence of metric measure spaces}

We recall here two notions of convergence for metric measure spaces. The first one is the  \emph{pointed measured Gromov convergence} of (equivalence classes of) metric measure spaces (or pmG for short), introduced in \cite{GigMonSav}. In particular, we use what they call ``extrinsic notion'' of convergence. This is used to characterise rectifiability and controls the Hausdorff distance of large subsets (in a measure theoretic sense) of the underlying metric spaces. The second notion is the \emph{pointed measured Gromov-Hausdorff convergence} of metric measure spaces (or pmGH for short), which requires the convergence both of the measures and of the underlying metric spaces. The latter notion is stronger, and the two coincide for uniformly doubling metric measure spaces (see Lemma \ref{lem:convergence} for the precise statement).

Recall that a pointed metric measure space consists of a metric measure space together with a distinguished point. Two pointed metric measure spaces $(\X_i,\di_i,\m_i,x_i)$, $i = 1, 2$ are isomorphic if there exists an isometric embedding $\iota : \supp(\m_1) \to \X_2$ such that 
\begin{equation}
    \iota_\#(\m_1) = \m_2 \qquad\text{and}\qquad \iota(x_1)=x_2.
\end{equation}
The isomorphism classes of pointed metric measure spaces are denoted by $\mathcal X=[(\X,\di,\m,x)]$. 
In the following, we use the notation $\weakto$ to denote the weak convergence of measures, in duality with continuous functions with bounded support. 

\begin{definition}[pmG convergence]
    Let $\{(\X_n, \di_n, \m_n, x_n)\}_{n\in\N\cup\{\infty\}}$ be a sequence of pointed metric measure spaces. Then, the corresponding sequence of equivalence classes $n\mapsto\mathcal X_n$ converges to $\mathcal X_{\infty}=[(\X_\infty,\di_\infty,\m_\infty,x_\infty)]$  \emph{in the pointed measured Gromov sense} (pmG, for short) if there exists a complete and separable metric space $(\Z, \di^\Z)$ and isometric embeddings $\iota_n : \X_n \to \Z$, for every $n\in \N\cup\{\infty\}$, such that 
    \begin{equation}
        (\iota_n)_\#\m_n \weakto (\iota_\infty)_\#\m_\infty\qquad \text{and}\qquad\iota_n(x_n) \to \iota_\infty(x_\infty)\in \supp((\iota_\infty)_\#\m_\infty)\ \text{ in }\Z.    
    \end{equation}
    The triple $\{(\Z, \di^\Z),\iota_n\}$ is called an effective realisation of the convergence.  
\end{definition}

\begin{definition}[pmGH convergence]
\label{def:pmGH_convergence}
Let $\{(\X_n, \di_n,\m_n, x_n)\}_{n\in\N\cup\{\infty\}}$ be a sequence of pointed metric measure spaces. We say that $(\X_n, \di_n,\m_n, x_n)$ converges to $(\X_\infty, \di_\infty,\m_\infty, x_\infty)$ in the \emph{pointed measured Gromov-Hausdorff sense} (pmGH for short) if there exists a separable metric space $(\Z,\di^\Z)$ and isometric
 embeddings $\iota_n : (\supp(\m_n),\di_n) \to (\Z,\di^\Z)$, for every $n\in\N\cup\{\infty\}$, such that: for every $\varepsilon > 0$ and $R > 0$,  there exists
 $n_0\in \mathbb{N}$ such that for all $n > n_0$, 
 \begin{equation}
    \iota_\infty(B^{\X_\infty}(x_\infty, R)) \subset B^\Z_\varepsilon[\iota_n(B^{\X_n}(x_n,R))] \qquad\text{and}\qquad \iota_n(B^{\X_n}(x_n,R))\subset B^\Z_\varepsilon[\iota_\infty(B^{\X_\infty}(x_\infty,R))],
 \end{equation}
 where $B^\Z_\varepsilon[A] := \{z \in \Z : \di^\Z(z,A) < \varepsilon\}$ for every subset $A \subset \Z$, and
 \begin{equation}
 \label{eq:convergence_measures}
     (\iota_n)_\#\m_n \weakto (\iota_\infty)_\#\m_\infty, \qquad\text{as }n\to\infty.
 \end{equation}
\end{definition}

\begin{remark}
    The classical pointed Gromov-Hausdorff convergence (pGH for short) of a sequence of metric spaces can be extracted from the above definition, dropping the requirement \eqref{eq:convergence_measures}. Clearly, if $(\X_n, \di_n,\m_n, x_n)\to(\X_\infty, \di_\infty,\m_\infty, x_\infty)$ in the pmGH sense, then $(\X_n, \di_n, x_n)\to(\X_\infty, \di_\infty, x_\infty)$ in the pGH sense.  
\end{remark}

\begin{remark}[Comparison with Bate's convergence]
    In \cite[Def.\ 4.9]{MR4506771}, Bate introduces the distance between pointed metric measure spaces $d_*$. This distance induces a convergence which is in general weaker than the pmGH convergence, since it does not require the underlying sets to be close in the pGH sense. Nonetheless, it implies the pGH convergence of subsets of large measure. Moreover, as a consequence of \cite[Prop.\ 2.20]{MR4506771}, the convergence induced by $d_*$ agrees with the pmG convergence, cf. \cite[Rmk.\ 4.14]{MR4506771}. 
\end{remark}

In general, the pmGH convergence implies the pmG convergence. See \cite[Prop.\ 3.30]{GigMonSav}. The converse is also true in uniformly doubling metric measure spaces. We recall below the precise statement. See \cite[Prop.\ 3.33]{GigMonSav}.

\begin{lemma}
\label{lem:convergence}
    Let $\{(\X_n, \di_n, \m_n, x_n)\}_{n\in\N\cup\{\infty\}}$ be a sequence of pointed metric measure spaces. Assume that 
    \begin{enumerate}
        \item[i)] $[(\X_n, \di_n, \m_n, x_n)] \to [(\X_\infty,\di_\infty,\m_\infty,x_\infty)]$ in the pmG sense;
        \item[ii)] $\supp(\m_\infty) = \X_\infty$;
        \item[iii)] The spaces $(\X_n, \di_n, \m_n)$ are uniformly globally doubling, with uniformly bounded (with respect to $n$) doubling constant.
    \end{enumerate}
    Then $(\X_n, \di_n, \m_n, x_n) \to (\X_\infty,\di_\infty,\m_\infty,x_\infty)$ in the pmGH sense.
\end{lemma}

\subsection{Metric measure tangents}

The notions of convergence of pointed metric measure space can be used to define tangent spaces. In a metric measure space $(\X,\di,\m)$, for every $x\in\supp(\m)$ and $r\in(0,1)$,  define the rescaled measure $\m_r^x$ as
\begin{equation}
\label{eq:tangent_measure}
    \m_r^x:= \left(\int_{B(x,r)}1-\frac1r\di(\cdot,x)\de\m\right)^{-1}\m.
\end{equation}

 \begin{definition}[The collection of tangent spaces]
    Let $(\X,\di,\m)$ be a metric measure space, let $x\in\supp(\m)$ and $r\in(0,1)$. Consider the rescaled and normalised pointed metric measure space $(\X, \di_r,\m^x_r,x)$, where $\di_r:=r^{-1}\di$. Then, the collection of pmG tangents is 
    \begin{equation}
        \Tan_{\rm pmG}(\X,\di,\m,x):=\Big\{\mathcal Y=[({\sf Y},\di^\Y, \n^\Y, y)]: \exists\,r_i\downarrow 0\text{ s.t. }[(\X, \di_{r_i},\m^x_{r_i},x)]\xrightarrow[{\rm pmG}]{i\to \infty} \mathcal Y\Big\}.
    \end{equation}
    Similarly, the collection of pmGH tangents is 
    \begin{equation}
        \Tan_{\rm pmGH}(\X,\di,\m,x):=\Big\{({\sf Y},\di^\Y, \n^\Y, y): \exists\,r_i\downarrow 0\text{ s.t. }(\X, \di_{r_i},\m^x_{r_i},x)\xrightarrow[{\rm pmGH}]{i\to \infty} ({\sf Y},\di^\Y, \n^\Y, y)\Big\}.
    \end{equation}
 \end{definition}

\begin{remark}
    In Lemma \ref{lem:convergence}, the assumption for the sequence to be uniformly globally doubling with a universal constant is used to get precompactness in the pmGH topology. Thus, when considering tangents, we can weaken the above assumption by simply requiring that the original measure is locally doubling.
\end{remark}

We show now that the tangent measure of a measure with finite and positive lower and upper densities, is locally Ahlfors regular. Consider a metric space $(\X, \di)$, a Borel measure $\mu$ on $\X$ and a constant $n>0$. The $n$-th lower and upper $\mu$-densities of a Borel set $E\subset\X$ at $x\in\X$ are defined respectively as 
\begin{equation}
\label{eq:def_ul_densities}
        \Theta_*^n[\mu](E,x):= \liminf_{r\to 0}\frac{\mu(B(x,r)\cap E)}{r^{n}}\qquad\text{and}\qquad\Theta^*_n[\mu](E,x):= \limsup_{r\to 0}\frac{\mu(B(x,r)\cap E)}{r^{n}}.
    \end{equation}
To simplify the notation, in the following, we will use the shorthands $\Theta_*^n[\mu](x) =\Theta_*^n[\mu](\X,x)$ and $\Theta^*_n[\mu](x) =\Theta^*_n[\mu](\X,x)$. Note that, by definition, for every $E\subset\X$, it holds that $$\Theta_*^n[\mu](E,\cdot), \Theta^*_n[\mu](E,\cdot)\in[0,\infty] \quad \text{and}\quad  \Theta_*^n[\mu](E,\cdot)\leq \Theta^*_n[\mu](E,\cdot).$$
Moreover,  $\mu$ is said to be locally $n$-Ahlfors regular if, for every compact $K\subset \X$, there exist $C\geq 1$ and $r_0\in (0,\diam\X]$, depending on $K$, such that
    \begin{equation}
    \label{eq:local_ahlfors}
         C^{-1}r^{n}\leq \mu(B(x,r))\leq C r^{n},\qquad\forall\,x\in K,\ r\in (0,r_0).
    \end{equation}

\begin{prop}
\label{prop:tan=tan}
    Let $(\X, \di,\m)$ be a metric measure space, with $\m$ locally uniformly doubling, cf.\ \eqref{eq:loc_unif_doubling}. Let $n>0$ and assume that $0<\loden{n}(x)\leq \upden{n}(x)<\infty$, for $\m$-a.e.\ $x\in\X$.
    Then, for $\m$-a.e.\ $x\in \X$ and for every $\mathcal Y\in \Tan_{{\rm pmG}}(\X, \di,\m,x)$, there exists a representative $(\Y,\di^{\Y},\n^{\Y},y)\in \mathcal Y$ such that
    \begin{itemize}
        \item[i)] $\supp(\n^\Y)=\Y$;
        \item[ii)] $(\Y,\di^\Y,\n^\Y,y)\in\Tan_{{\rm pmGH}}(\X, \di,\m,x)$;  
        \item[iii)]$\n^\Y$ is locally $n$-Ahlfors regular.
    \end{itemize}
\end{prop}

    Before proceeding with the proof we observe that, given an isometry $\iota: (\X,\di) \to (\Z,\di^\Z)$ and a Borel measure $\m$ on $(\X,\di)$, the very definition of pushforward measure guarantees that
    \begin{equation}\label{eq:pushforward&isometries}
        \m\big(B^\X(x,r)\big) = \iota_\# \m\big(B^{\Z}(\iota(x),r)\big), \qquad \forall\, x\in \X,\, r>0.
    \end{equation}

\begin{proof}
    Take any $x\in\X$ such that $0<\loden{n}(x)\leq \upden{n}(x)<\infty$ and the conclusion of \cite[Thm.\ 3.2]{MR3377394} holds. Consider any $\mathcal Y \in \Tan_{{\rm pmG}}(\X, \di,\m,x)$. Then, there exist a sequence of radii $r_i\downarrow 0$, as $i\to\infty$, such that $[(\X, \di_{r_i},\m^x_{r_i},x)] \to \mathcal Y$ in the pmG sense. 
    
    Firstly, we show that there exists a representative of $\mathcal Y$ that satisfies items {\slshape i)} and {\slshape ii)}. Since $\m$ is locally uniformly doubling, it has full support and the sequence $\{(\X,\di_{r_i},\m_{r_i}^x,x)\}_{i\in \N}$ is precompact in the pmGH topology. Thus, up to considering a subsequence of $\{r_i\}_{i\in \N}$, there exists a metric measure space
    $(\Y,\di^\Y,\n^\Y,y)\in \mathcal Y$ such that 
    \begin{equation}
        (\X,\di_{r_i},\m_{r_i}^x,x)\xrightarrow[{\rm pmGH}]{i\to\infty}(\Y,\di^\Y,\n^\Y,y).
    \end{equation}
     In particular, this means that there exist isometries
    \begin{equation}
        \iota_i: (\X, \di_{r_i})\to (\Z,\di^\Z) \qquad \text{and} \qquad \iota_\infty: (\text{supp}(\n^\Y), \di^{\Y})\to (\Z,\di^\Z),
    \end{equation}
    such that
    \begin{equation}\label{eq:relisationofpmGH}
        \iota_i(x) \to \iota_\infty (y) \quad \text{and} \quad (\iota_i) _\# \m_{r_i}^x \weakto (\iota_\infty) _\# \n^{\Y}, \qquad \text{as }i \to \infty.
    \end{equation}
     By the properties of $\m$, the measure $\n^\Y$ is locally uniformly doubling and has full support. 
    
    Secondly, we prove that $\n^\Y$ is locally $n$-Ahlfors regular on $\Y$. For ease of notation, set $\n:=\n^\Y$ and $\di_i:=\di_{r_i}=r_i^{-1}\di$, in the rest of the proof. Note that, by our choice of $x$, there exist constants $C=C(x)\geq 1$ and $r_0=r_0(x)>0$ such that 
    \begin{equation}
    \label{eq:pointwise_ahlfors}
        C^{-1}r^{n}\leq \m(B(x,r))\leq Cr^{n},\qquad \forall\,r\in (0,r_0).
    \end{equation}
    Fix $\rho>0$ and $\varepsilon>0$. Then, using \eqref{eq:pushforward&isometries} and \eqref{eq:relisationofpmGH}, we get that 
    \begin{equation}
    \label{eq:upper_estimate_nball}
        \begin{split}
            \n\big(B^\Y(y, \rho)\big) &= (\iota_\infty) _\# \n \big( B^\Z(\iota_\infty (y), \rho)\big) \leq \liminf_{i\to \infty} \,(\iota_i) _\# \m_{\rho_i}^x \big( B^\Z(\iota_\infty (y), \rho)\big) \\
            &\leq \liminf_{i\to \infty} \,(\iota_i) _\# \m_{r_i}^x \big( B^\Z(\iota_i (x), \rho(1+\varepsilon))\big) = \liminf_{i\to \infty} \, \m_{r_i}^x \big( B^{\di_i}(x, \rho(1+\varepsilon))\big) \\
            &= \liminf_{i\to \infty} \, \left(\int_{B(x,r_i)}1-\frac1{r_i}\di(\cdot,x)\de\m\right)^{-1} \m\big( B(x, r_i \rho (1+\varepsilon))\big),
        \end{split}
    \end{equation}
    where the first inequality follows from the lower semicontinuity of the evaluation on bounded open sets with respect to the weak convergence of measures. Using \eqref{eq:pointwise_ahlfors}, we deduce that for every $r_i$ sufficiently small 
    \begin{equation} 
        \m\big( B(x, r_i \rho (1+\varepsilon))\big) \leq C r_i^{n} \rho^{n} (1+\varepsilon)^{n}
    \end{equation}
    and
    \begin{equation}
    \label{eq:lower_bound_rescaling}
        \int_{B(x,r_i)}1-\frac1{r_i}\di(\cdot,x)\de\m \geq \int_{B(x,r_i/2)} \frac 12 \de\m = \frac 12 \m \big(B(x,r_i/2)\big) \geq \frac 12 C^{-1} \bigg( \frac {r_i
}2\bigg)^n.
    \end{equation}
    Therefore, combining these two inequality with \eqref{eq:upper_estimate_nball}, we conclude that 
    \begin{equation}
        \n\big(B^\Y(y, \rho)\big) \leq  C^22^{n+1} (1+\varepsilon)^n \, \rho^n, \qquad \forall\, \rho>0.
    \end{equation}
    Proceeding in a similar way, for any $\varepsilon>0$ we obtain that 
    \begin{equation}
        \begin{split}
            \n\big(B^\Y(y, \rho)\big) &= (\iota_\infty) _\# \n \big( B^\Z(\iota_\infty (y), \rho)\big) \geq (\iota_\infty) _\# \n \big( \bar B^\Z(\iota_\infty (y), \rho(1-\varepsilon))\big)\\
            &\geq \limsup_{i\to \infty} \,(\iota_i) _\# \m_{r_i}^x \big( \bar B^\Z(\iota_\infty (y), \rho(1-\varepsilon))\big) \geq \limsup_{i\to \infty} \,(\iota_i) _\# \m_{r_i}^x \big( B^\Z(\iota_i (x), \rho(1-\varepsilon)^2)\big) \\
            &= \limsup_{i\to \infty} \, \m_{r_i}^x \big( B^{\di_i}(x, \rho(1-\varepsilon)^2)\big) \\
            &= \limsup_{i\to \infty} \, \left(\int_{B(x,r_i)}1-\frac1{r_i}\di(\cdot,x)\de\m\right)^{-1} \m\big( B(x, r_i \rho (1-\varepsilon)^2)\big),
        \end{split}
    \end{equation}
    where this time we used the upper semicontinuity of the evaluation on bounded closed sets with respect to the weak convergence of measures. Using \eqref{eq:pointwise_ahlfors}, we obtain that for $r_i$ sufficiently small
     \begin{equation} 
        \m\big( B(x, r_i \rho (1+\varepsilon))\big) \geq C^{-1} r_i^
{n} \rho^{n} (1-\varepsilon)^{2n}
    \end{equation}
    and
    \begin{equation}
        \int_{B(x,r_i)}1-\frac1{r_i}\di(\cdot,x)\de\m \leq \m \big(B(x,r_i)\big) \leq
 C r_i^n.
    \end{equation}
    We can then deduce that 
    \begin{equation}
        \n\big(B^\Y(y, \rho)\big) \geq \frac{ (1-\varepsilon)^{2n}}{C^2} \, \rho^n, \quad \forall \rho>0.
    \end{equation}
    In conclusion, we have found two constants $\tilde c, \tilde C>0$ such that 
    \begin{equation}
         \tilde c \rho^n \leq \n \big(B^\Y(y,\rho)\big) \leq \tilde C \rho^n, \qquad \forall\, \rho >0.
    \end{equation}
    Moreover, we observe that these constants do not depend on the specific element $(\Y, \di^\Y, \n, y)$ of $\Tan_{{\rm pmGH}}(\X,\di,\m,x)$, but only on the point $x\in\X$. 
    
    Now, let $w\in \Y$. By \cite[Thm.\ 3.2]{MR3377394}, we have that $(\Y, \di^\Y, \n^{w}_1, w)\in \Tan_{{\rm pmGH}}(\X,\di,\m,x)$. In particular, the first part of the proof guarantees that
    \begin{equation}
    \label{eq:iteration_ahlfors_reg}
          \tilde c \rho^n \leq \n^{w}_1 \big(B^\Y(w,\rho)\big) \leq  \tilde C \rho^n, \qquad \forall \rho >0.
    \end{equation}
    We introduce the function $\textbf{n}: \Y \to \R$, defined as  
    \begin{equation}
        \textbf{n}(z):=  \int_{B(z,1)}1-\di(\cdot,x)\de\n,
    \end{equation}
    so that $\n^{w}_1= \textbf{n}(w)^{-1}\cdot  \n$. On the one hand, observe that 
    \begin{equation}
        \textbf{n}(z)\leq \n \big( B^\Y(z,1)\big) \leq \n \big( B^\Y (y, \di^\Y(y,z) +1)\big) \leq \tilde C (\di^\Y(y,z) +1)^{n},
    \end{equation}
    thus $\textbf{n}$ is locally bounded above. On the other hand, Fatou's Lemma ensures that $\textbf{n}$ is lower semicontinuous, and therefore, being always positive, it is locally bounded from below by a positive constant. Now, let $K\subset \Y$ be a compact subset containing $w$ and let 
    \begin{equation}
        k_1:=\min_{z\in K} \textbf n(z)>0 \qquad\text{and}\qquad k_2:=\max_{z\in K} \textbf n(z)<\infty.
    \end{equation}
    Then, using \eqref{eq:iteration_ahlfors_reg} and the definition of $\n_1^w$, we obtain 
    \begin{equation}
        k_1 \tilde c \rho^{n} \leq \textbf{n}(w)\n_1^w\big(B^\Y(w,\rho)\big)=\n\big(B^\Y(w,\rho)\big)\leq k_2\tilde C \rho^{n},\qquad \forall\, \rho>0.
    \end{equation}
    This proves that the measure $\n$ is locally $n$-Ahlfors regular.
\end{proof}

\subsection{Rectifiability of \texorpdfstring{$\cd(K,N)$}{CD(K,N)} spaces with unique tangents} \label{sec:rectCD}

In this section, we prove the first main result. We start by recalling two fundamental theorems that are at the core of the strategy. The first one is a result by Le Donne \cite{MR2865538} (after the work of Berestovski\u{\i} \cite{MR985283}) that describes tangents of metric spaces with unique tangents. The second one is a characterization of rectifiability in terms of metric tangents  due to Bate \cite{MR4506771}.

\begin{theorem}[{\cite[Thm.\ 1.2]{MR2865538}}]
\label{thm:eld}
     Let $(\X,\di,\m)$ be a geodesic metric measure space, with $\m$ a doubling measure. Assume that, for $\m$-a.e.\ $x \in \X$, the set $\Tan_{{\rm pGH}}(\X, \di,x)$ contains only one element. Then, for $\m$-a.e.\ $x \in \X$, the element in $\Tan_{{\rm pGH}}(\X, \di,x)$ is a \sF Carnot group. 
\end{theorem}

\begin{remark}
    In the above theorem, the assumption on the measure can be relaxed, by requiring instead that $\m$ is locally uniformly doubling. 
\end{remark}

In the following, given $K \geq 1$, we denote by $\mathcal L_K$ the set of isometry classes of all pointed proper metric spaces $(\X,\di,x)$ for which there exists a surjective and $K$-bi-Lipschitz map $\psi: (\R^n,\normdotc) \to \X$. Also, we define the set
\begin{equation}
     \widetilde{\mathcal L}_K = \{\mathcal X=[(\X,\di,\m,x)] : (\supp\m,\di) \in \mathcal L_K\}.
\end{equation}

\begin{theorem}[{\cite[Thm.\ 1.2]{MR4506771}}]
\label{thm:bate}
    Let $(\X,\di)$ be a complete metric space, $n \in \N$ and let $E \subset X$ be Borel measurable with $\mathcal H^n(E) < \infty$. Then, the following are equivalent: 
    \begin{enumerate}
        \item[i)] $E$ is $n$-rectifiable;
        \item[ii)] for $\mathcal H^n$-a.e.\ $x \in E$, $\Theta^n_*[\mathcal H^n](E,x)> 0$ and there exists an $n$-dimensional Banach space $(\R^n,\normdotc_x,x)$ such that
    \begin{equation}
    \label{eq:bate}
        \Tan_{{\rm pmG}}(\X, \di,\mathcal H^n_{|E},x) = \{[(\R^n,\normdotc_x,c_x\mathcal H^n,0)]\},
    \end{equation}
    where $c_x>0$ is an explicit constant;
        \item[iii)] for $\mathcal H^n$-a.e.\ $x \in E$, $\Theta^n_*[\mathcal H^n](E,x)> 0$ and there exists  $K_x \geq 1$ such that
        \begin{equation}
            \Tan_{{\rm pmG}}(\X, \di,\mathcal H^n_{|E},x) \subset \widetilde{\mathcal L}_{K_x}.
        \end{equation}
    \end{enumerate} 
\end{theorem}

\begin{remark}
    Compared with \cite{MR4506771}, we use a different normalisation to define the tangent measure, cf.\ \eqref{eq:tangent_measure}. This justifies the constant $c_x$ in \eqref{eq:bate}, which is different from the one of \cite[Thm.\ 1.2]{MR4506771}.
\end{remark}

Given a metric measure space $(\X,\di,\m)$ and $n>0$, recall the definition of $n$-th upper and lower densities of $\m$, cf.\ \eqref{eq:def_ul_densities}, and the subsequent shorthands. Note that, if $(\X,\di)$ is a geodesic space and $\m$ is locally uniformly doubling, for every fixed $r>0$, the function $\X\ni x\mapsto \m(B(x,r))$ is continuous (see for example \cite[Lem.\ 3.8]{MR3165282}) and thus $x\mapsto \loden{n}(x)$ and $x\mapsto \upden{n}(x)$ are Borel functions.

\begin{lemma}
\label{lem:gmt}
    Let $(\X,\di,\m)$ a geodesic metric measure space with $\m$ locally uniformly doubling and $\m(\X)<\infty$. Fix $n>0$ and assume that $0<\loden{n}(x)\leq \upden{n}(x)<\infty$, for $\m$-a.e.\ $x\in\X$. Then, for every $\varepsilon>0$, there exists a Borel set $A_\varepsilon\subset \X$ such that 
    \begin{itemize}
        \item[i)] $\m(\X\setminus A_\varepsilon)<\varepsilon$;
        \item[ii)] there exists a constant $C_\varepsilon\geq 1$ such that, for every $x\in A_\varepsilon$, there is a $r_{\varepsilon,x}>0$ so that 
        \begin{equation}
        \label{eq:this_is_not_ahlfors_reg}
            C_\varepsilon^{-1}r^n\leq \m(B(x,r))\leq C_\varepsilon r^n,\qquad\forall\, r\in [0,r_{\varepsilon,x}];
        \end{equation}
        \item[iii)] $\m_{|A_\varepsilon}$ is equivalent to $\mathcal H^n_{|A_\varepsilon}$.
    \end{itemize}
\end{lemma}

\begin{proof}
    Fix $\varepsilon>0$. For every $k\in\N$, define the set 
    \begin{equation}
        A_k:=\left\{x\in\X:\tfrac1k<\loden{n}(x)\leq \upden{n}(x)<k\right\}.  
    \end{equation}
    By construction $A_k\subset A_{k+1}$ and, by the assumption on the densities, up to an $\m$-negligible set, $\X = \bigcup_{k\in\N}A_k$. Moreover, since $(\X,\di)$ is geodesic, the sets $A_k$ are also Borel. Therefore, choosing $k_\varepsilon\in\N$ sufficiently large, we find a set $A_\varepsilon:=A_{k_\varepsilon}\subset X$ satisfying item {\slshape i)}. Now, by construction, for $x\in A_\varepsilon$, we have 
    \begin{equation}
        \frac1{k_\varepsilon}<\liminf_{r\to 0}\frac{\m(B(x,r))}{r^n}\leq\limsup_{r\to 0}\frac{\m(B(x,r))}{r^n}<k_\varepsilon.
    \end{equation}
    Hence, given $\delta>0$, there exists $r_{\varepsilon,x}>0$ (depending on $\delta$) such that 
    \begin{equation}
        \frac1{k_\varepsilon}-\delta<\frac{\m(B(x,r))}{r^n}<k_\varepsilon+\delta,\qquad\forall\,r\in (0,r_{\varepsilon,x}].
    \end{equation}
    This shows item {\slshape ii)}. In addition, by the continuity $x\mapsto\m(B(x,r))$, we observe that the function $x\mapsto r_{\varepsilon,x}$ is measurable. We proceed to the proof of the last item. Consider a Borel set $E\subset A_\varepsilon$ and let $O\subset \X$ be any open set containing $E$. Given $\rho>0$, for every $x\in E$, let $\rho_x>0$ be sufficiently small so that $\rho_x\leq \min\{r_{\varepsilon,x},\rho\}$ and $B(x,\rho_x)\subset O$. Then, since $\{B(x,\rho_x)\}_{x\in E}$ is a covering of $E$ and $(\X,\di)$ is separable, there is a countable refinement i.e.\ $E\subset \bigcup_{i\in\N}B(x_i,\rho_i)$ with $\rho_i:=\rho_{x_i}$. By Vitali covering lemma, we may find $I\subset\N$ such that $B(x_i,\rho_i)\cap B(x_j,\rho_j)=\emptyset$ for every $i\neq j$ in $I$ and $E\subset\bigcup_{i\in I}B(x_i,5\rho_i)$. Then, we have: 
    \begin{equation}
        \mathcal H_{5\rho}^n(E)\leq \sum_{i\in I} \left[\diam\left(B(x_i,5\rho_i)\right)\right]^n \leq10^n\sum_{i\in I} \rho_i^n\leq 10^nC_\varepsilon \sum_{i\in I}\m(B(x_i,\rho_i))\leq 10^nC_\varepsilon\m(O).
    \end{equation}
    Letting $\rho\to 0$, we obtain that $ \mathcal H^n(E)\leq C(\varepsilon,n)\m(O)$. Since $O$ is an arbitrary open set containing $E$, we conclude that $ \mathcal H^n(E)\leq C(\varepsilon,n)\m(E)$.
    
    For the converse inequality, fix $\tilde\varepsilon>0$. Reasoning as in the first part of the proof, we find a constant $\rho_{\tilde\varepsilon}>0$ and Borel set $B_{\tilde\varepsilon}\subset A_\varepsilon$ with the property that $\m(A_\varepsilon\setminus B_{\tilde\varepsilon})<\tilde\varepsilon$ and $r_{\varepsilon,x}\geq \rho_{\tilde\varepsilon}$ for every $x\in B_{\tilde\varepsilon}$. Let $\rho>0$ be sufficiently small, so that $\rho<\min\{\rho_{\tilde\varepsilon},\diam(\X)/2\}$. We observe that, by construction, for every $x\in E\cap B_{\tilde\varepsilon}$, we have
    \begin{equation}
         \m(B(x,r))\leq C_\varepsilon r^n,\qquad\forall\, r\in [0,\rho].
    \end{equation}
    Now, let $\{B(x_i,\rho_i)\}_{i\in\N}$ a countable covering of $E\cap B_{\tilde\varepsilon}$, with $\rho_i\leq\rho$. Then, we deduce that
    \begin{equation}
        \m(E\cap B_{\tilde\varepsilon})\leq \sum_{i=1}^\infty \m(B(x_i,\rho_i)) \leq C_\varepsilon \sum_{i=1}^\infty \rho_i^n \leq C_\varepsilon \sum_{i=1}^\infty \left[\diam\left(B(x_i,\rho_i)\right)\right]^n.
    \end{equation}
    In the last inequality we used that $\rho_i \leq \diam\left(B(x_i,\rho_i)\right)$, which holds because $\rho_i \leq \diam(\X)/2$ and $(\X,\di)$ is geodesic. Since the covering was arbitrary and recalling that the Hausdorff measure and the spherical Hausdorff measure are equivalent, we finally obtain that
    \begin{equation}
        \m(E\cap B_{\tilde\varepsilon})\leq C'(\varepsilon,n)\mathcal H^n_\rho(E\cap B_{\tilde\varepsilon}).
    \end{equation}
    Letting $\rho\to0$, we find that $\m(E\cap B_{\tilde\varepsilon})\leq C'(\varepsilon,n)\mathcal H^n(E\cap B_{\tilde\varepsilon})\leq C'(\varepsilon,n)\mathcal H^n(E)$. To conclude, observe that, by construction of $B_{\tilde\varepsilon}$, 
    \begin{equation}
        \m(E)=\m(E\cap B_{\tilde\varepsilon})+\m(E\setminus B_{\tilde\varepsilon})\leq \m(E\cap B_{\tilde\varepsilon})+\tilde\varepsilon\leq C'(\varepsilon,n)\mathcal{H}^n(E)+\tilde\varepsilon.
    \end{equation}
    Since $\tilde\varepsilon>0$ is arbitrary, we conclude the proof.
\end{proof}

\begin{remark}
    Note that the assumption $\m(\X)<\infty$ is technical and, in the sequel, does not play a role. Indeed, we only deal with local properties of the reference measure, hence we can always restrict ourselves to a finite-measure subset of $\X$ by $\sigma$-compactness. 
\end{remark}

\begin{theorem}
\label{thm:rectifiability<5}
    Given $K\in\R$ and $N\in (1,\infty)$, let $(\X,\di,\m)$ be a $\cd(K,N)$ space. Assume that  
    \begin{itemize}
        \item[i)] for $\m$-a.e.\ $x\in\X$, $\Tan_{{\rm pGH}}(\X, \di,x)$ contains a single element, up to isometry;
        \item[ii)] there exists $n\in(0,5)$ such that $0<\loden{n}(x)\leq \upden{n}(x)<\infty$, for $\m$-a.e.\ $x\in\X$. 
    \end{itemize}
     Then:
     \begin{itemize}
 \item     $n$ is an integer;
 \item $(\X,\di,\m)$ is $(\m,n)$-rectifiable;
 \item for $\m$-a.e.\ $x\in\X$, $\Tan_{{\rm pmGH}}(\X, \di,\m,x)$ contains a single element (up to isomorphism) and its metric part is isometric to an $n$-dimensional Banach space.
 \end{itemize}
\end{theorem}

\begin{proof}
    Without loss of generality, we may assume that $\supp(\m)=\X$. Firstly, $\m$ is a locally finite measure on a proper metric space, hence $\m$ is $\sigma$-finite. Thus, there exists a countable family $\{A_i\}_{i\in\N}$ of Borel sets, such that 
    \begin{equation}
        \X=\bigcup_{i\in\N}A_i,\qquad\m(A_i)<\infty, \text{ for every }i\in\N. 
    \end{equation}
    Secondly, given $i,k\in \N$, since $\m$ has positive and finite $n$-densities and $(\X,\di)$ is geodesic, we find a Borel subset $A_i^k\subset A_i$ satisfying items {\slshape i)-iii)} of Lemma \ref{lem:gmt}, with $\varepsilon=1/k$.  In addition, by inner regularity of $\m$, we may also assume that $A_i^k$ is closed. As a consequence of Lemma \ref{lem:gmt}, $\m$ is equivalent to the $n$-dimensional Hausdorff measure on $A_i^k$, the lower Hausdorff density $\Theta_*^n
[\mathcal H^n](A_i^k,\cdot)$ is positive $\mathcal H^n$-a.e., and $\mathcal H^n(A_i^k)>0$. Thus, we are in position to apply {\slshape iii)}  $\Rightarrow$ {\slshape i)} of Theorem \ref{thm:bate} to the set $A_i^k$, for every $i,k\in\N$. To conclude rectifiability, we ought to check that almost every tangent of $A_i^k$ is a finite-dimensional Banach space. 
    
    We claim that, for $\m$-a.e.\ $x\in\X$ and every $(\Y,\di^\Y,\n^\Y,y)\in\Tan_{{\rm pmGH}}(\X, \di,\m,x)$, $(\Y,\di^\Y)$ is isometric to a finite-dimensional Banach space. Firstly, as $(\X,\di,\m)$ is a $\cd(K,N)$ space, $\m$ is locally uniformly doubling. In addition, the assumption {\slshape i)} ensures that at $\m$-a.e.\ point $\Tan_{{\rm pGH}}(\X, \di,x)$ contains a single element. Thus, we may apply Theorem \ref{thm:eld} and deduce that there exists $\mathcal N\subset\X$ of vanishing $\m$-measure such that
    \begin{equation}
        \Tan_{{\rm pGH}}(\X, \di,x)=\{(G^x,\di^x_{SF},\e_x)\},\qquad\text{for every }x\in\X\setminus\mathcal N,
    \end{equation} 
    where $(G^x,\di^x_{SF})$ is a \sF Carnot group. Secondly, let $x\in\X\setminus\mathcal N$ and let $(\Y,\di^\Y,\n^\Y,y)\in\Tan_{{\rm pmGH}}(\X, \di,\m,x)$. Note that, by definition of pmGH convergence, $(\Y,\di^\Y,y)$ is isometric to $(G^x,\di^x_{SF},\e_x)$. Up to taking $x$ in the complement of a bigger $\m$-null set $\mathcal N$, following the proof of items {\slshape i)} and {\slshape iii)} of Proposition \ref{prop:tan=tan}, we find that $\n^\Y$ has full support and it is a locally $n$-Ahlfors regular measure. This implies that the Haudorff dimension of $(G^x,\di^x_{SF})$ is equal to $n \in (0,5)$. Therefore, $n=\mathcal Q^x$, where $\mathcal Q^x$ is the homogeneous dimension of $G^x$ defined in Remark \ref{rmk:hom_dim}, and it is an integer. Thirdly, according to \cite[Cor.\ 4.8]{proceedingsVarenna}, $\n^\Y$ is a bounded measure and we can apply \cite[Thm.\ 4.12]{proceedingsVarenna} to conclude that
    \begin{equation}
        \Tan_{{\rm pmGH}}(G^x,\di^x_{SF},\n^\Y,g)=\{(G^x,\di^x_{SF},\hat c_x\mathcal H^{n},\e_x)\},\qquad\text{for $\n^\Y$-a.e.\ }g\in G,
    \end{equation}
    where $\hat c_x>0$ is an explicit constant which depends only on $(G^x,\di^x_{SF})$. 
    At this point, we apply \cite[Thm.\ 3.2]{MR3377394}, stating that ``Tangents of tangents are tangents''. More precisely, we deduce that
    \begin{equation}
        (G^x,\di^x_{SF},\hat c_x\mathcal H^{n},\e_x) \in \Tan_{{\rm pmGH}}(\X,\di,\m,x),\qquad\text{for $\m$-a.e.\ }x \in \X. 
    \end{equation}
    As $(\X,\di,\m)$ is a $\cd(K,N)$ space, by stability of the $\cd$ condition, we obtain that $(G^x,\di^x_{SF},\hat c_x\mathcal H^{n})$ is a $\cd(0,N)$ space. Recalling that $G^x$ has homogeneous dimension $n\in(0,5)$, we can combine Corollary \ref{cor:classification} and Theorem \ref{thm:nocd_hei} to conclude that $G^x$ must be commutative and $n=\dim_{\rm Top}G^x$, meaning that $(G^x,\di^x_{SF})\cong (\R^n,\normdotc_x)$. This finally shows that $(\Y,\di^\Y)$ is isometric to the finite-dimensional Banach space $(\R^n,\normdotc_x)$, as claimed.
    
    The next step is to describe the set $\Tan_{{\rm pmG}}(\X, \di,\mathcal{H}^{n},x)$ in terms of the pmG tangents to $(\X,\di,\m)$ at a suitably chosen point $x\in\X$. Since $\mathcal H^n$ is equivalent to $\m$ on the set $A_i^k$, there exists a function $\rho\in L^1_{loc}(A_i^k,\m)$ such that $\mathcal H^n=\rho \m$ and $C_{1/k}^{-1}\leq\rho\leq C_{1/k}$, $\m$-a.e.\ on $A_i^k$, where $C_{1/k}\geq 1$ is determined by item {\slshape ii)} of Lemma \ref{lem:gmt}. Hence, let $x\in(\X\setminus \mathcal N)\cap A_i^k$ such that it is a density point of $A_i^k$ with respect to both $\m$ and $\mathcal H^n$ and it is a $\m$-Lebesgue point of $\rho$. Note that, as $\m$ and $\mathcal H^n$ are equivalent, we are discarding a $\m$-negligible subset of $A_i^k$. 
    
    We claim that 
    \begin{equation}
    \label{eq:tan_haus=tan_m}
        \Tan_{{\rm pmG}}(\X, \di,\mathcal{H}^{n}_{|A_i^k},x)=\Tan_{{\rm pmG}}(\X, \di,\m_{|A_i^k},x).
    \end{equation}
    Firstly, observe that since $\mathcal H^n_{|A_i^k}=\rho \m_{|A_i^k}$, then, for every $r>0$, the rescaled measures at $x$ defined as in \eqref{eq:tangent_measure} satisfy
    \begin{multline}
    \label{eq:rescaled_haus}
        \big[\mathcal H^n_{|A_i^k}\big]_r^x=\left(\int_{B(x,r)\cap A_i^k}1-\frac1r\di(\cdot,x)\de\mathcal H^n\right)^{-1}\mathcal H^n_{|A_i^k}\\
        =\left(\int_{B(x,r)\cap A_i^k}1-\frac1r\di(\cdot,x)\de\m\right)
        \left(\int_{B(x,r)\cap A_i^k}\Big(1-\frac1r\di(\cdot,x)\Big)\rho\de\m\right)^{-1}\rho\big[\m_{|A_i^k}\big]_r^x.
    \end{multline}
    In particular, $\big[\mathcal H^n_{|A_i^k}\big]_r^x$ is absolutely continuous with respect to $\m_r^x$, with density given by $h_r^x\rho$, where $h_r^x$ is the constant defined by \eqref{eq:rescaled_haus}. By our choice of the base point $x$, we have that $h_r^x\to 1/\rho(x)$ as $r\to 0^+$. Secondly, let $\mathcal Y=[(\Y,\di^\Y,\n^\Y,y)]\in \Tan_{{\rm pmG}}(\X, \di,\m_{|A_i^k},x)$. Then, by definition, there exist a sequence $\{r_j\}_{j\in\N\cup\{\infty\}}$ converging to $0$ and an effective realisation $\{(\Z,\di^Z),\iota_j\}$ such that 
    \begin{equation}
    \label{eq:n_tan_meas}
        (\iota_j)_\#\big[\m_{|A_i^k}\big]_{r_j}^x \weakto (\iota_\infty)_\#\n^\Y\qquad \text{and}\qquad\iota_j(x) \to \iota_\infty(y)\in \supp((\iota_\infty)_\#\n^\Y)\ \text{ in }\Z.    
    \end{equation}
    Then, Claim \eqref{eq:tan_haus=tan_m} is verified if we show that, for every $\varphi\in C_{bs}(\Z)$, it holds
    \begin{equation}
    \label{eq:claim_tan_measure_haus}
        \int_\Z \varphi\de(\iota_j)_\#\big[\mathcal H^n_{|A_i^k}\big]_{r_j}^x\xrightarrow{j\to\infty} \int_\Z \varphi\de (\iota_\infty)_\#\n^\Y.
    \end{equation}
    Note that, by \eqref{eq:rescaled_haus}, we have 
    \begin{equation}
        \int_\Z \varphi\de(\iota_j)_\#\big[\mathcal H^n_{|A_i^k}\big]_{r_j}^x=h_{r_j}^x\int_\Z \varphi\de\big(\rho(\iota_j)_\#\big[\m_{|A_i^k}\big]_{r_j}^x\big)=h_{r_j}^x\int_\X \rho(\varphi\circ\iota_j)\de\big[\m_{|A_i^k}\big]_{r_j}^x.
    \end{equation}
    Furthermore, a straightforward estimate shows that 
    \begin{multline*}
        \left|h_{r_j}^x\int_\X \rho(\varphi\circ\iota_j)\de\big[\m_{|A_i^k}\big]_{r_j}^x- \int_\X (\varphi\circ \iota_\infty)\de\n^\Y\right|\\ \leq \left|h_{r_j}^x\int_\X (\rho-\rho(x))(\varphi\circ\iota_j)\de\big[\m_{|A_i^k}\big]_{r_j}^x\right|+\left|\rho(x)h_{r_j}^x\int_\X(\varphi\circ\iota_j)\de\big[\m_{|A_i^k}\big]_{r_j}^x- \int_\X( \varphi\circ \iota_\infty)\de\n^\Y\right|.
    \end{multline*}
    The second term in the right-hand side of the inequality converges to $0$ by \eqref{eq:n_tan_meas} and since $\rho(x)h_{r_j}^x\to 1$. Hence, the convergence \eqref{eq:claim_tan_measure_haus} is valid provided that 
    \begin{equation}
        \label{eq:refined_claim}
        \int_\X (\rho-\rho(x))(\varphi\circ\iota_j)\de\big[\m_{|A_i^k}\big]_{r_j}^x\xrightarrow{j\to\infty}0.
    \end{equation}
    The function $\varphi$ has bounded support, thus there exists $R>0$ such that $\supp(\varphi)\subset B^\Z(\iota_j(x),R)$ for every $j$ sufficiently large. Then, $\varphi=\chi_{B^\Z(\iota_j(x),R)}\varphi$ and 
    \begin{equation}
    \begin{split}
        \int_\X (\rho-\rho(x))(\varphi\circ\iota_j)\de\big[\m_{|A_i^k}\big]_{r_j}^x &=\int_\X (\rho-\rho(x))((\chi_{B^\Z(\iota_j(x),R)}\varphi)\circ\iota_j)\de\big[\m_{|A_i^k}\big]_{r_j}^x \\ 
        &=\int_{B^{\di_{r_j}}(x,R)} (\rho-\rho(x))(\varphi\circ\iota_j)\de\big[\m_{|A_i^k}\big]_{r_j}^x\\
        &=\int_{B^\di(x,r_jR)} (\rho-\rho(x))(\varphi\circ\iota_j)\de\big[\m_{|A_i^k}\big]_{r_j}^x.      
    \end{split}
    \end{equation}
    Therefore, for a suitable constant $C>0$, we can estimate
    \begin{multline}
    \label{eq:final_estimate}
        \left|\int_\X (\rho-\rho(x))(\varphi\circ\iota_j)\de\big[\m_{|A_i^k}\big]_{r_j}^x\right| \leq \int_\X\left|(\rho-\rho(x))(\varphi\circ\iota_j)\right|\de\big[\m_{|A_i^k}\big]_{r_j}^x\\ \leq C\|\varphi\|_{L^\infty(\Z)}\fint_{B^\di(x,r_jR)}\left|(\rho-\rho(x))\right|\de\m_{|A_i^k},
    \end{multline}
    where in the last inequality, we used item {\slshape ii)} of Lemma \ref{lem:gmt} and \eqref{eq:lower_bound_rescaling}. Finally, the right-hand side of \eqref{eq:final_estimate} converges to $0$ since $x$ was chosen to be an $\m$-Lebesgue point of $\rho$. This shows \eqref{eq:refined_claim} and, consequently, the convergence in \eqref{eq:claim_tan_measure_haus}. Thus, we conclude the proof of Claim \eqref{eq:tan_haus=tan_m}.

    Lastly, since the base point $x$ was chosen to be a density point for $A_i^k$ with respect to $\m$ and $\mathcal H^n$, applying \cite[Cor.\ 5.7]{MR4506771} and by Claim \eqref{eq:tan_haus=tan_m}, we have that 
    \begin{equation}
    \label{eq:tan=tan=tan=tan}
        \Tan_{{\rm pmG}}(\X, \di,\mathcal{H}^{n},x)=\Tan_{{\rm pmG}}(\X, \di,\mathcal{H}^{n}_{|A_i^k},x)=\Tan_{{\rm pmG}}(\X, \di,\m_{|A_i^k},x)=\Tan_{{\rm pmG}}(\X, \di,\m,x).
    \end{equation} 
    Furthermore, letting $\mathcal Y\in \Tan_{{\rm pmG}}(\X, \di,\m,x)$, by Proposition \ref{prop:tan=tan}, there exists a representative $(\Y,\di^\Y,\n^\Y,y)\in \mathcal Y$ belonging to $\Tan_{{\rm pmGH}}(\X, \di,\m,x)$ and such that $\n^\Y$ has full support. As the metric side of the tangent is independent on the measure, by the first part of the proof, we conclude that $(\Y,\di^\Y)$ is isometric to $(\R^n,\normdotc_x)$. Thus, there exists $K_x\geq 1$ for which $\mathcal Y\in \widetilde{\mathcal L}_{K_x}$. Finally, applying Theorem \ref{thm:bate}, we conclude that $A_i^k$ is $n$-rectifiable, for every $i,k\in\N$. Since the sets $\{A_i^k\}_{k\in\N}$ exhausts $A_i$ up to a $\m$-negligible set, we conclude that $A_i$ has positive $\mathcal H^n$-measure and it is $n$-rectifiable up to a $\m$-negligible set for every $i\in \N$. Therefore, we conclude that $\X$ is $(\m,n)$-rectifiable, in the sense of Definition \ref{def:rect2}. 
    
    By Theorem \ref{thm:bate}, we also deduce that for $\m$-a.e.\ $x\in \X$
    \begin{equation}
    \label{eq:final_tan}
        \Tan_{{\rm pmG}}(\X, \di,\mathcal H^{n},x) = \{[(\R^{n},\normdotc_x,c_x\mathcal H^{n},0)]\}.
    \end{equation}
    For $\m$-a.e.\ $x\in\X$ for which \eqref{eq:final_tan} holds, there exist $i,k\in\N$ such that $x\in A_i^k$ and \eqref{eq:tan=tan=tan=tan} holds. We deduce that
    \begin{equation}
        \Tan_{{\rm pmGH}}(\X, \di,\m,x) = \{(\R^{n},\normdotc_x,c_x\mathcal H^{n},0)\},
    \end{equation}
    up to isomorphism.
\end{proof}

The strategy developed in the proof Theorem \ref{thm:rectifiability<5} is quite flexible and works for every dimension $n$. We formalise this in Theorem \ref{thm:meta_rectifiability}, whose statement is given assuming the validity of the following conjecture. This conjecture is supported by various works in sub-Riemannian and \sF geometry (see in particular \cite{Juillet2020,MR4562156,MR4623954} for the sub-Riemannian case and \cite{magnabosco2023failure,borza2024measure,borza2024curvatureexponentsubfinslerheisenberg} for the \sF case). 

\begin{conj}
\label{conj:carnot_nocd}
    Let $(G,\di_{SF})$ be a \sF Carnot group of homogeneous dimension $\mathcal{Q}$. Assume that the metric measure space $(G,\di_{SF},\mathcal H^\mathcal{Q})$ is a $\cd(0,N)$ space for some $N > 1$. Then, $G$ is a finite-dimensional Banach space, i.e.\ $G$ is commutative, $\g=V_1$ and $\mathcal{Q}=\dim_{\rm Top}G$.
\end{conj}

\begin{theorem}
\label{thm:meta_rectifiability}
    Given $K\in\R$ and $N\in (1,\infty)$, let $(\X,\di,\m)$ be a $\cd(K,N)$ space. Assume that  
    \begin{itemize}
        \item[i)] for $\m$-a.e.\ $x\in\X$, $\Tan_{{\rm pGH}}(\X, \di,x)$ contains a single element, up to isometry;
        \item[ii)] there exists $n>0$ such that $0<\loden{n}(x)\leq \upden{n}(x)<\infty$, for $\m$-a.e.\ $x\in\X$. 
    \end{itemize}
    If Conjecture \ref{conj:carnot_nocd} holds, then $n$ is an integer and $(\X,\di,\m)$ is $(\m,n)$-rectifiable. Moreover, for $\m$-a.e.\ $x\in\X$, $\Tan_{{\rm pmGH}}(\X, \di,\m,x)$ contains a single element (up to isomorphism) and its metric part is isometric to a $n$-dimensional Banach space.
\end{theorem}

\begin{proof}
     Following the proof of Theorem \ref{thm:rectifiability<5}, we see that, for $\m$-a.e.\ $x\in \X$, and for every $(\Y, \di^\Y,\n^\Y,x)\in\Tan_{{\rm pmGH}}(\X, \di,\m,x)$, the metric space $(\Y,\di^\Y)$ is isometric to a \sF Carnot group $(G^x,\di_{SF}^x)$ with $\dim_{\text{Haus}}G^x=n$. Moreover, we get that for $\m$-a.e.\ $x\in \X$ 
    \begin{equation}
        (G^x,\di^x_{SF},\hat c_x\mathcal H^{N},\e_x) \in \Tan_{{\rm pmGH}}(\X,\di,\m,x). 
    \end{equation}
     Therefore, combining Conjecture \ref{conj:carnot_nocd} with the stability of the $\cd$ condition, we deduce that $G^x$ is commutative for $\m$-a.e.\ $x\in\X$. Proceeding as in the proof of Theorem \ref{thm:meta_rectifiability}, we obtain the desired conclusion.
\end{proof}

\begin{remark}
    \label{rmk:stratification}
    The techniques employed to prove Theorems \ref{thm:rectifiability<5} and \ref{thm:meta_rectifiability} extends to the case when there is a stratification, namely when the assumption {\slshape ii)} is replaced by the following: 
    {\slshape 
    \begin{itemize}
        \item[ii')] there exist $n>0$ and a countable collection of $\m$-measurable sets $\{R_{k_i}\}_{i\in \N}$, covering $\X$ up to an $\m$-negligible set, such that $0<k_i\leq n$ and
            \begin{equation}
                0<\loden{k_i}(x)\leq\upden{k_i}(x),\qquad \text{for $\m$-a.e.\ }x\in R_{k_i}.
            \end{equation}
    \end{itemize}}
    \noindent In this case, we would conclude that there are only a finite number of strata, indexed by integers $k\in\{1,\ldots,\lfloor n\rfloor\}$, each of which is $(\m,k)$-rectifiable. For example, the $\cd(K,N)$ space built in \cite{breaking} satisfies the assumptions {\slshape i)} and {\slshape ii')}, and it has indeed rectifiable strata.

    The main reason for not stating the theorems in this generality is because, at the moment, we do not have an operative way to define a stratification. This is in contrast with $\RCD$ spaces, where a stratification can be defined using the splitting theorem, see \cite{MR3945743}.
\end{remark}

\section{Rectifiability of non-collapsed \texorpdfstring{$\MCP(K,N)$}{MCP(K,N)} spaces with unique tangents}
\label{sec:rectifiability_non-collapsed_mcp}

In this last section we prove the second main result, namely Theorem \ref{thm:rectifiability_MCP}, showing the rectifiability of non-collapsed $\MCP(K,N)$ spaces having unique metric tangent almost everywhere. The proof of Theorem \ref{thm:rectifiability_MCP} relies on the general strategy developed in Section \ref{sec:rectCD} and on a new result (Theorem \ref{thm:noMCP}) about the failure of the measure contraction property in \sF Carnot groups.

Let $(G,\di_{SF})$ be a \sF Carnot group. Recall (see the discussion around \eqref{eq:stratification}) that $\g := T_\e G$ admits a stratification $\g = V_1 \oplus\ldots\oplus V_\s$, where  $\s\geq 1$ is the step of the group. According to \cite[Prop.\ 2.9]{EeroELD}, there are canonical sub-Finsler structures on $G/[G,G] \cong V_1$ and $G/\exp(V_\s)$ such that the projections
\begin{equation}
    \pi:G \to G/[G,G]\qquad \text{and} \qquad \pi_{\s-1}:G \to G/\exp(V_\s)
\end{equation}
are submetries. In the following, we denote by $(G/[G,G],\di'_{SF})$ and $(G/\exp(V_\s),\di''_{SF})$ the resulting \sF Carnot groups. Observe that, since $G/[G,G] \cong V_1$ is commutative, the \sF distance $\di'_{SF}$ is a norm distance and the metric space $(G/[G,G],\di'_{SF})$ is a finite-dimensional Banach space.

The next result corresponds to {\cite[Thm.\ 3.2]{EeroELD}}. The statement is slightly modified for our purposes, but it can be shown following the proof of {\cite[Thm.\ 3.2]{EeroELD}}.

\begin{theorem}\label{thm:HLD}
    Let $(G, \di_{SF})$ be a sub-Finsler Carnot group of step $\s$ and let $\gamma:[0,1] \to G$ be a geodesic. Then, there exist constants $C>0$ and $\delta>0$ such that for every $t,s\in [0,\delta]$ it holds that
    \begin{equation}
        \di_{SF}(\gamma(s),\gamma(t)) - C \di_{SF}(\gamma(s),\gamma(t))^\frac{\s}{\s-1} \leq \di''_{SF}(\pi_{\s-1}\circ\gamma(s),\pi_{\s-1}\circ\gamma(t))\leq \di_{SF}(\gamma(s),\gamma(t)).
    \end{equation}
\end{theorem}

\begin{lemma}\label{lem:closetoGeo}
    Let $(\Z,\di)$ be a locally compact metric space and fix $\bar z\in \Z$. For every $\theta>0$, there exists $\rho>0$ with the following property: given any continuous curve $\gamma\in C([0,1],\Z)$ with $\gamma(0)=\bar z$ such that 
    \begin{equation}
        (1-\rho)\ell |t_2-t_1|\leq \di(\gamma(t_1),\gamma(t_2)) \leq (1+\rho)\ell |t_2-t_1| \qquad \forall\, t_1,t_2\in [0,1],
    \end{equation}
    for some $\ell\in[0,1]$, then there exists $\xi\in \Geo(\Z)$ with $\xi(0)= \bar z$ and $\xi(1)\in \bar B(\bar z,1)$, such that $\sup_{t\in[0,1]}\di(\gamma(t),\xi(t))<\theta$.
\end{lemma}

\begin{proof}
    Assume by contradiction this is not the case. Then, there exist a sequence of continuous curves $\{\gamma_i\}_{i\in \N}\subset C([0,1],\Z)$ and a sequence   $\{\rho_i\}_{i\in \N}\subset \R_+$ converging to $0$  such that
    \begin{equation}\label{eq:almostgeodesic}
        (1-\rho_i)\ell_i  |t_2-t_1|\leq \di(\gamma_i(t_1),\gamma_i(t_2)) \leq (1+\rho_i)\ell_i |t_2-t_1|\qquad \forall\, t_1,t_2\in [0,1],
    \end{equation}
    for some $\ell_i\in[0,1]$, and $\sup_{t\in[0,1]}\di(\gamma_i(t),\xi(t))\geq \theta$ for every $\xi\in \Geo(\Z)$ with $\xi(0)= \bar z$ and $\xi(1)\in \bar B(\bar z,1)$. Up to taking a subsequence of indexes and using the Arzelà-Ascoli theorem, we can suppose that $\gamma_i \to \bar\gamma\in C([0,1],\Z)$ uniformly and $\ell_i\to \bar \ell\in [0,1]$. Passing to the limit in \eqref{eq:almostgeodesic} we deduce that $\bar \gamma$ is a geodesic of length $\bar \ell$, and such that $\bar\gamma(0)=\bar z$ and $\bar\gamma(1)\in \bar B(\bar z,1)$. Thus, for $i$ sufficiently large, $\sup_{t\in[0,1]}\di(\gamma_i(t),\bar\gamma(t))< \theta$, giving a contradiction.
\end{proof}

\begin{theorem}\label{thm:noMCP}
     Let $(G, \di_{SF})$ be a sub-Finsler Carnot group with homogeneous dimension $N$. Then, either $G$ is commutative or $(G, \di_{SF}, \mathcal{H}^N)$ does not satisfy the $\MCP(0,N)$ condition. 
\end{theorem}

\begin{proof}
    If $G$ is commutative there is nothing to prove. If $G$ is not commutative, assume  by contradiction that $(G, \di_{SF}, \mathcal{H}^N)$ satisfies the $\MCP(0,N)$ condition. Let $\delta_\e$ be the Dirac delta measure centered at $\e$, let $\mathcal{H}^N(\bar B(\e,1))^{-1}\mathcal{H}^N|_{\bar B(\e,1)}\in \ProbTwo(G)$ and consider the $W_2$-Wasserstein geodesic $\{\mu_t\}_{t\in [0,1]}$ connecting them. Then, by $\MCP(0,N)$, for every $t\in[0,1]$, it holds that
    \begin{equation}
    \label{eq:mcp_def2}
        \frac{1}{\mathcal{H}^N(\bar B(\e,1))}\mathcal{H}^N\geq t^N \mu_t.
    \end{equation}
    Recall that, by homogeneity of $\mathcal H^N$ with respect to the dilations of $G$, we have 
    \begin{equation}
    \label{eq:hom_HN}
        \mathcal{H}^N \big( \bar B(\e, r)\big) = \mathcal{H}^N(\bar B(\e,1)) \cdot r^N, \qquad \forall r>0.
    \end{equation}
    Then, since $\mu_t$ is supported on $\bar B(\e,t)$, combining \eqref{eq:mcp_def2} with \eqref{eq:hom_HN}, we deduce that
    \begin{equation}
        \mu_t \leq  \frac{1}{t^N\mathcal{H}^N(\bar B(\e,1))}\mathcal{H}^N|_{\bar B(\e,t)}= \frac{1}{\mathcal{H}^N(\bar B(\e,t))}\mathcal{H}^N|_{\bar B(\e,t)}. 
    \end{equation}
    As both the right-hand side and the left-hand side are probability measures, then the equality must hold, that is 
    \begin{equation}\label{eq:rigiditymut}
        \mu_t= \frac 1{ \mathcal{H}^N(\bar B(\e,t))} \mathcal{H}^N|_{\bar B(\e,t)}, \qquad \forall t\in (0,1].
    \end{equation} 
    Now consider any $\eta \in \Prob(\Geo(G))$ inducing the Wasserstein geodesic $\{\mu_t\}_{t\in [0,1]}$, meaning that $(e_t)_\#\eta=\mu_t$ for every $t\in[0,1]$. By definition, $\eta$ is supported on geodesics $\gamma:[0,1]\to G$ with $\gamma(0)=\e$ and  $\gamma(1)\in \bar B(\e,1)$, thus having length at most $1$. 
    
    For every $s\in (0,1]$, we define the $s$-rescaling of $\eta$ as 
    \begin{equation}
        \eta_s= (\res_0^s)_\# (\delta_{1/s})_\# \eta,
    \end{equation}
    where the map $\res_0^s:C([0,1],G)\to C([0,1],G)$ is defined as 
    \begin{equation}
        \res_0^s (\gamma) (t) = \gamma(st),\qquad\forall\,t\in[0,1].
    \end{equation}
    Observe that, as a consequence of \eqref{eq:rigiditymut},
    \begin{equation}\label{eq:rigiditymut2}
        (e_t)_\#\eta_s= (e_{st})_\# (\delta_{1/s})_\# \eta= (\delta_{1/s})_\#(e_{st})_\#  \eta =\frac 1{ \mathcal{H}^N(\bar B(\e,t))}  \mathcal{H}^N|_{\bar B(\e,t)}, \qquad \forall t\in (0,1].
    \end{equation}
    Consider a sequence $\{s_i\}_{i\in \N}$ converging to $0$. Up to subsequences, Prokhorov theorem ensures that $\eta_{s_i} \weakto \bar \eta \in \Prob(\Geo(G))$. Passing to the limit in \eqref{eq:rigiditymut2}, we get
    \begin{equation}
        (e_t)_\#\bar\eta= \frac 1{ \mathcal{H}^N(\bar B(\e,t))}\mathcal{H}^N|_{\bar B(\e,t)}, \qquad \forall t\in (0,1].
    \end{equation}
    
    \medskip
    
    \noindent\textbf{Claim 1.} The measure $\bar \eta$ is concentrated on the set 
    \begin{equation}
        P := \big\{ \gamma \in \Geo(G) \,:\,\gamma(0)=\e,\,\gamma(1)\in\bar B(\e,1)\text{ and } \pi_{\s-1}\circ\gamma \in \Geo(G/\exp(V_\s))\big\}. 
    \end{equation} 

    \medskip 
    
    \noindent Fix two constants $\varepsilon, \theta>0$. For every $C,\delta>0$, define the measurable set 
    \begin{equation}
        R_{C,\delta}:= \{ \gamma \in \supp(\eta) \, :\, \text{Theorem \ref{thm:HLD} holds with $C$ and $\delta$}\}. 
    \end{equation}
    By Theorem \ref{thm:HLD}, we clearly have that 
    \begin{equation}
        \supp(\eta) = \bigcup_{n\in \N} R_{n,1/n}.
    \end{equation}
    In addition, for $n\leq m$, it holds that $R_{n,1/n}\subset R_{m,1/m}$.
    Therefore, there exists $\bar n\in\N$ such that $\eta(R_{\bar n,1/\bar n})> 1-\varepsilon$. For every $s\in (0,1]$, consider the set 
    \begin{equation}
        R_{\bar n,1/\bar n}^s:= \{ \res_0^s \circ \,\delta_{1/s}(\gamma) \,:\,\gamma \in R_{\bar n,1/\bar n}\} \subset \Geo(G).
    \end{equation}
    Observe that, by the definition of $\eta_s$, it holds that $\eta_s(R_{\bar n,1/\bar n}^s)>1-\varepsilon$. Now, noting that $R_{\bar n,1/\bar n}^s$ only contains geodesics of length at most $1$, we can apply Lemma \ref{lem:closetoGeo} and find $\rho$ corresponding to the $\theta$ we fixed at the beginning. Let $s>0$  be sufficiently small such that $s<1/\bar n$ and $\bar n s ^\frac{1} {\s-1}<\rho$.  Let $\tilde\gamma_s \in R^s_{\bar n,1/\bar n}$. Then there exists $\gamma \in R_{\bar n,1/\bar n}$ so that $\tilde \gamma_s= \res_0^s \circ\, \delta_{1/s}(\gamma)$; therefore,  for every $t_1,t_2\in[0,1]$, it holds that
    \begin{equation}
    \begin{split}
        \di''_{SF}\big(\pi_{\s-1}\circ\tilde\gamma_s(t_1),\pi_{\s-1}\circ\tilde\gamma_s(t_2)\big)  &=\di''_{SF}\big(\delta_{1/s}(\pi_{\s-1}\circ\gamma(st_1)),\delta_{1/s}(\pi_{\s-1}\circ\gamma(st_2))\big)  \\
        &=\frac{1}{s}\di''_{SF}\big(\pi_{\s-1}\circ\gamma(st_1),\pi_{\s-1}\circ\gamma(st_2)\big) \\
        &\geq \frac 1s \Big[\di_{SF}\big(\gamma(s t_1),\gamma(st_2)\big) - \bar n \di_{SF}\big(\gamma(s t_1),\gamma(st_2)\big)^\frac{\s}{\s-1} \Big] \\
        &\geq\di_{SF}\big(\delta_{1/s}(\gamma(s t_1)),\delta_{1/s}(\gamma(st_2))\big)\big(1 - \bar n s ^\frac{1} {\s-1}|t_2-t_1|^\frac{1} {\s-1}\big)\\
        &\geq \di_{SF}\big(\tilde\gamma_s( t_1)),\tilde\gamma_s(t_2)\big)(1 - \rho),
    \end{split}
    \end{equation}
    where, in the first inequality, we used that $\gamma \in R_{\bar n,1/\bar n}$.
    On the other hand, as $\pi_{\s-1}$ is a submetry, we have that for every $t_1,t_2\in[0,1]$
    \begin{equation}
        \di''_{SF}\big(\pi_{\s-1}\circ\tilde\gamma_s(t_1),\pi_{\s-1}\circ\tilde\gamma_s(t_2)\big)  \leq 
        \di_{SF}\big(\tilde\gamma_s( t_1)),\tilde\gamma_s(t_2)\big).
    \end{equation}
    Thus, applying Lemma \ref{lem:closetoGeo}, we find $\xi\in \Geo(G/\exp(V_\s))$ with $\xi(0)= \e$ and $\xi(1)\in \bar B(\e,1)$, such that $\sup_{t\in[0,1]}\di(\tilde\gamma_s(t),\xi(t))<\theta$. Since $\tilde\gamma_s$ was an arbitrary curve in $R_{\bar n,1/\bar n}^s$, we have proved that, for every $s>0$ sufficiently small,  
    \begin{equation}
        R_{\bar n,1/\bar n}^s \subset \bigg\{\sigma\in \Geo(G) \,:\, \inf_{\substack{\xi\in \Geo(G/\exp(V_\s))\\ \xi(0)=\e,\,\, \xi(1)\in \bar B(\e,1)}} \sup_{t\in[0,1]}\di(\sigma(t),\xi(t))\leq\theta \bigg \}=: P_\theta.
    \end{equation} 
    It follows that, for $i$ sufficiently large, $\eta_{s_i} (P_\theta)\geq \eta_{s_i} (R_{\bar n,1/\bar n}^{s_i})>1-\varepsilon$. Since, by definition, $P_\theta$ is a closed set in $C([0,1],G)$, we conclude that 
    \begin{equation}
        \bar \eta (P_\theta) \geq \limsup_{i\to\infty} \eta_{s_i} (P_\theta) >1-\varepsilon.
    \end{equation}
    As $\varepsilon$ and $\theta$ were arbitrary and $P= \cap_{\theta>0} P_\theta$, we conclude that $\bar \eta (P) = 1$. 
    
    \medskip
    
    If the step $\s$ is larger than $2$, we can reiterate this blow-up procedure until we find $\tilde \eta\in \Prob(\Geo(G))$ concentrated on the set 
    \begin{equation}
        \tilde P := \big\{ \gamma \in \Geo(G) \,:\,\gamma(0)=\e,\,\gamma(1)\in\bar B(\e,1)\text{ and } \pi\circ\gamma \in \Geo(G/[G,G])\big\} 
    \end{equation}
    and such that 
    \begin{equation}
    \label{eq:mu_t}
        (e_t)_\#\tilde\eta= \frac 1{\mathcal{H}^N(\bar B(\e,t))} \mathcal{H}^N|_{\bar B(\e,t)}, \qquad \forall t\in (0,1].
    \end{equation}

    \medskip
    
    \noindent \textbf{Claim 2.}
    For every $t\in[0,1]$, it holds that $\mathcal{H}^N\big(\bar B(\e,t) \setminus e_t(\tilde P)\big)>0$.

    \medskip
    
    Using the exponential map of the group, we can identify the Lie group $(G,\cdot)$ with $(\R^n,\star)$, where $n= \dim V_1 + \dots + \dim V_\s$ and the group operation $\star$ is defined using the Baker-Campbell-Hausdorff formula. In this identification, we decompose $G \cong \mathcal V_1 \oplus \mathcal V_2$ where $\mathcal V_1=V_1$ and $\mathcal V_2=V_2 \oplus \cdots \oplus V_\s$ and we consider the (component-wise) projection maps $\p_{\mathcal V_1}:\R^n\to \mathcal V_1$ and  $\p_{\mathcal V_2}:\R^n\to \mathcal V_2$. Observe that, in exponential coordinates, $G/[G,G]\cong \mathcal V_1$ as normed spaces and $\pi=\p_{\mathcal{V}_1}$. Moreover, for every $g\in G$, its coordinate representation in $\R^n$ is $\log(g)$ and thus we can identify $\dis_g$ as the linear subspace of $\R^n$ given by $(\hat L_{\log(g)})_{*,0}V_1$, where $\hat L_v$ is the left-multiplication by $v$ in $(\R^n,\star)$. Finally, we endow $\mathcal V_2$ with the Euclidean norm, denoted by $\normdot'$, and we recall that $\mathcal V_1$ is naturally endowed with the norm $\normdot$ inducing the sub-Finsler structure.
    
    Let $i:\mathcal V_1\to\R^n$ be the canonical inclusion of $\mathcal V_1$ in $\R^n$ and observe that by definition $p_{\mathcal V_1}\circ i=\text{Id}$. Then, for every $g\in G$, we see that $p_{\mathcal V_1}\circ (\hat L_{\log(g)})_{*,0}\circ i=p_{\mathcal V_1}\circ i=\text{Id}$. Indeed, given any $v\in\mathcal V_1$, by the Baker-Campbell-Hausdorff formula, we obtain that 
    \begin{equation}
        (\hat L_{\log(g)})_{*,0}(v)=\left.\frac{\de}{\de t}\right|_{t=0} \log(g)\star tv = v+w, 
    \end{equation}
    where $w\in\mathfrak g^2=[\mathfrak g,\mathfrak g]$. Hence $p_{\mathcal V_1}((\hat L_{\log(g)})_{*,0}(v))=p_{\mathcal V_1}(v+w)=v$. This shows that $\p_{\mathcal V_1}|_{\dis_{g}}$ is invertible with inverse given by $(\p_{\mathcal V_1}|_{\dis_{g}})^{-1}=(\hat L_{\log(g)})_{*,0}\circ i$. In addition, we easily deduce that the map $(\p_{\mathcal V_1}|_{\dis_{g}})^{-1}$ depends smoothly on $g$.
    
    Then, by compactness, there exist a constant $K>0$ such that 
    \begin{equation}\label{eq:uniformvertical}
        \norm{\p_{\mathcal V_2}\circ (\p_{\mathcal V_1}|_{\dis_{g}})^{-1}  (v)}' \leq K \norm{v}, \qquad \forall\, g\in \bar B(\e,1),\, v\in \mathcal V_1.
    \end{equation}
    
    Now, consider any $\gamma\in \tilde P$ and call $\tilde \gamma= \pi\circ\gamma\in \Geo(G/[G,G])$. In particular, $\tilde \gamma$ is absolutely continuous and $\norm{\tilde \gamma'(t)}= l$ for $\Leb^1$-a.e.\ $t\in [0,1]$, where $l$ is the length of $\tilde \gamma$ in $G/[G,G]\cong (\mathcal V_1,\normdot)$. The curve $\gamma$ is itself absolutely continuous and we know that 
    \begin{equation}
        \tilde\gamma'(t) = \pi_* (\gamma'(t)) = \pi(\gamma'(t)),\qquad\text{for } \Leb^1\text{-a.e.\ } t\in [0,1],
    \end{equation}
    where in the second identity we used the linearity of $\pi$ in exponential coordinates. 
    Therefore, since $\pi=\p_{\mathcal{V}_1}$ and $\gamma'(t)\in \dis_{\gamma(t)}$, we have that
    \begin{equation} 
        \gamma'(t)= (\p_{\mathcal V_1}|_{\dis_{\gamma(t)}})^{-1}  (\tilde\gamma'(t)), \qquad\text{for } \Leb^1\text{-a.e.\ } t\in [0,1].
    \end{equation}
    Applying the estimate \eqref{eq:uniformvertical}, we see that 
    \begin{equation}
        \norm{\p_{\mathcal V_2} (\gamma'(t))}' \leq K\norm{\tilde\gamma'(t)}=K l, \qquad\text{for } \Leb^1\text{-a.e.\ } t\in [0,1].
    \end{equation}
    In particular, being $\tilde \gamma$ a geodesic in a normed space, we have $\norm{\tilde\gamma(t)}=lt$ for every $t\in[0,1]$. Thus,
    \begin{equation}
        \norm{\p_{\mathcal V_2} (\gamma(t))}' \leq  \int_0^t \norm{\p_{\mathcal V_2} (\gamma'(s))}' \de s \leq Klt=K\norm{\tilde\gamma(t)},\qquad\forall\,t\in [0,1].
    \end{equation}
    Finally, we obtain that 
    \begin{equation}
        \text{Im} (\gamma) \subset \{(v_1,v_2)\in \mathcal{V}_1\oplus \mathcal{V}_2\,:\, \norm{v_2}'\leq K \norm{v_1}\}=: \mathscr{C}_K.
    \end{equation}
    As $\gamma$ was arbitrary, we conclude that for every $t\in[0,1]$
    \begin{equation}
        e_t(\tilde P) \subset \mathscr{C}_K \cap \bar B(\e,t).
    \end{equation}
    However, for every $t\in [0,1]$, the set $\bar B(\e,t) \setminus \mathscr{C}_K$ has positive measure and this concludes the proof of Claim 2.
    
    \medskip
    
    \noindent In conclusion,  for every $t\in[0,1]$, $(e_t)_\# \tilde\eta$ is concentrated on $e_t(\tilde P)$. However, \eqref{eq:mu_t} implies $(e_t)_\# \tilde\eta$ is supported on $\bar B(\e,t)$. This is in contradiction with the result of Claim 2. 
\end{proof}

\begin{remark}
Theorem \ref{thm:noMCP} is in line with other results in \sr and \sF setting studying the curvature exponent, i.e.\ the optimal dimensional parameter $N$ for which the measure contraction property $\MCP(K,N)$ holds for some $K$. Indeed, as a consequence of \cite[Prop.\ 5.49]{MR3852258} and \cite[Thm.\ D]{MR3852258} (see also \cite{MR3110060, MR3502622,MR3848070}), one sees that the curvature exponent of any equiregular \sr manifold (thus of every \sr Carnot group) is greater than or equal to the Hausdorff dimension plus one. An analogous result is proved for \sF Heisenberg groups in \cite{borza2024curvatureexponentsubfinslerheisenberg} (see also \cite{borzatashiro}). The statement of Theorem \ref{thm:noMCP} proves that that \emph{in every \sF Carnot group} the curvature exponent is strictly greater than the Hausdorff dimension.
\end{remark}

\begin{remark}[About the proof of Theorem \ref{thm:noMCP}]
\label{rmk:final_rmk}
(i) In order to contradict \eqref{eq:rigiditymut}, it would be sufficient to prove that, for some $t\in(0,1)$, the $t$-midpoint set 
\begin{align}
        M_t(\e,\bar B(\e,1)) := 
        \{
            x \in G \, : \, x = \gamma(t)
                 , \, 
            \gamma \in \Geo(G)
                 , \, 
            \gamma(0) =\e 
                \ \text{and} \  
            \gamma(1) \in \bar B (\e,1)
        \},
    \end{align}
    does not contain $\bar B(\e,t)$. Indeed, since $\eta\in \Prob(\Geo(G))$, we have that $\supp \mu_t \subset M_t(\e,\bar B(\e,1)) \subset \bar B(\e,t)$.
    
    It is possible to prove that, if the reference norm of the \sF structure on $G$ is strictly convex, then 
    \begin{equation}
        M_t(\e,\bar B(\e,1)) \subsetneq \bar B(\e,t).
    \end{equation}
    However, this is not true in general, as one can check in the so called quasi-Heisenberg group $\R\times \hei$. This is $G=\R^4$ equipped with the non-commutative group law, defined  by
\begin{equation}
    (w, x, y, z) \cdot (w', x', y', z') = \bigg(w+w', x+x',y+y',z+z'+\frac12(xy' - x'y)\bigg),
\end{equation}
for all $(w,x, y, z), (w',x', y', z')\in\R^4$,
with identity element $\e=(0,0,0,0)$. Define the left-invariant vector fields
\begin{equation}
    W:=\partial_w, \qquad X:=\partial_x-\frac{y}2\partial_z,\qquad Y:=\partial_y+\frac{x}2\partial_z.
\end{equation}
Then, $[W,X]=[W,Y]=0$, $[X,Y]=\partial_z=:Z$ and the Lie algebra of $G$ admits the stratification $\g=V_1\oplus V_2$,
where $V_1=\text{span}\{W,X,Y\}$ and $V_2:=\text{span}\{Z\}$. If we equip $V_1$ with the norm 
\begin{equation}
    \norm{(x_1,x_2,x_3)} := |x_1| + \sqrt{x_2^2+x^2_3},
\end{equation}
we obtain the \sF distance 
\begin{equation}
    \di_{SF}\big((w, x, y, z), (w', x', y', z')\big) = |w-w'| + \di_{SR}\big(( x, y, z), ( x', y', z')\big),
\end{equation}
where $\di_{SR}$ denotes the canonical \sr distance on the Heisenberg group. In the resulting \sF Carnot group $(G, \di_{SF})$, $M_t(\e,\bar B(\e,1)) = \bar B(\e,t)$ for every $t\in[0,1]$.

In conclusion, finding a contradiction to \eqref{eq:rigiditymut} is much easier if the reference norm of the \sF structure on $G$ is strictly convex. However, for the general case, a more elaborate strategy is necessary.\\
(ii) The strategy to prove Theorem \ref{thm:noMCP} is inspired by the main results of \cite{EeroELD} in a way that we are going to detail here. 

In a \sF Carnot group $(G,\di_{SF})$, given a Lipschitz curve $\gamma:[0,1] \to G$, for every $h\in(0,1]$, define the curve $\gamma_h: [0,1/h] \to G$ as
\begin{equation}
    \gamma_h(t):= \delta_\frac 1h \big( \gamma(\bar t)^{-1} \cdot \gamma(\bar t + ht) \big).
\end{equation}
Accordingly, introduce the collection of tangents to $\gamma$ at $0$ as the set
\begin{equation}
    \Tang(\gamma,0):=
 \{\sigma:\R \to G \, :\, \exists(h_j)_{j\in \N} \to 0, \, \gamma_{h_j}\to \sigma \text{ uniformly on compact sets}\}.
\end{equation}
The main result of \cite{EeroELD} states that for every $\sigma \in \Tang(\gamma,0)$, the curve $\pi\circ\sigma: \R \to G/\exp(V_\s)$ is a geodesic, i.e. $\sigma\in P$. Repeating this blow-up procedure, they also prove that for every $\tilde\sigma$ in the $(\s-1)$-times iterated tangents $\Tang^{\s -1
}(\gamma,\bar t)$ (see \cite{EeroELD} for the definition), the curve $\pi\circ\tilde\sigma: \R \to G/[G,G]$ is a geodesic,  i.e. $\tilde \sigma\in \tilde P$.

Our strategy relies on an analogous blow-up procedure which, however, has to be done for (almost) all geodesics in the support of $\eta$ simultaneously. The key observation is that the rescalings by a sufficiently small factor of all curves in a set of the type $R_{C,\delta}$ are definitely uniformly close to the set $P$. This is indeed remarkable, as the rescalings $\gamma_h$ of $\gamma\in \Geo(G)$ do not necessarily have a unique limit as $h\to 0$. 
\end{remark}

We have now all the ingredients to prove the second main result.

\begin{theorem}
\label{thm:rectifiability_MCP}
    Given $K\in\R$ and $N\in (1,\infty)$, let $(\X,\di,\m)$ be an $\MCP(K,N)$ space. Assume that  
    \begin{itemize}
        \item[i)] for $\m$-a.e.\ $x\in\X$, $\Tan_{{\rm pGH}}(\X, \di,x)$ contains a single element, up to isometry;
        \item[ii)] for $\m$-a.e.\ $x\in\X$, $0<\loden{N}(x)\leq \upden{N}(x)<\infty$. 
    \end{itemize}
    Then:
    \begin{itemize}
    \item $N$ is an integer;
    \item  $(\X,\di,\m)$ is $(\m,N)$-rectifiable;
    \item for $\m$-a.e.\ $x\in\X$, $\Tan_{{\rm pmGH}}(\X, \di,\m,x)$ contains a single element (up to isomorphism) and its metric part is isometric to a $N$-dimensional Banach space.
    \end{itemize}
\end{theorem}

\begin{proof}
     Following again the proof of Theorem \ref{thm:rectifiability<5}, we deduce that, for $\m$-a.e.\ $x\in \X$, and for every $(\Y, \di^\Y,\n^\Y,x)\in\Tan_{{\rm pmGH}}(\X, \di,\m,x)$, the metric space $(\Y,\di^\Y)$ is isometric to a \sF Carnot group $(G^x,\di_{SF}^x)$ with $\dim_{\text{Haus}}G^x=N$. Moreover, we get that for $\m$-a.e.\ $x\in \X$ 
    \begin{equation}
        (G^x,\di^x_{SF},\hat c_x\mathcal H^{N},\e_x) \in \Tan_{{\rm pmGH}}(\X,\di,\m,x). 
    \end{equation}
     Therefore, the stability of the measure contraction property, guarantees that $(G^x,\di^x_{SF},\hat c_x\mathcal H^{N})$ satisfies $\MCP(0,N)$. However, $G^x$ has homogeneous dimension $N$, as this is equal to its Hausdorff dimension, cf. Remark \ref{rmk:hom_dim}. Thus, Theorem \ref{thm:noMCP} allows to deduce that $G^x$ is commutative for $\m$-a.e.\ $x\in\X$. Proceeding as in the proof of Theorem \ref{thm:meta_rectifiability}, we obtain the desired conclusion.
\end{proof}

\begin{remark}\label{Rem:AfterThm1.2}
    The fact that Theorem \ref{thm:rectifiability_MCP} is proved assuming $\MCP(K,N)$ instead of $\cd(K,N)$ is quite remarkable. Indeed, according to \cite{MR4245620}, there exist plenty of \sr manifolds satisfying the measure contraction property. For instance, any Lipschitz Carnot group whose first layer is equipped with a left-invariant Riemannian metric and which is equipped with the Haar measure $\m$ satisfies $\MCP(0,N)$, for some $N >0$ (see \cite[Theorem 1.4]{MR4245620}). Moreover, such spaces have unique metric measure tangent at every point. However, it follows from \cite{Magnani} that these are not $(\m,n)$-rectifiable, for any choice of $n\in \mathbb{N}$. An explicit example is given by the Heisenberg group $\mathbb{H}^n=(\mathbb{R}^{2n+1}, \mathsf{d}_{SR}, \mathcal{L}^{2n+1})$, which satisfies $\MCP(0,N)$ for any $N\geq 2n+3$ (see \cite{MR2520783}), it has unique metric measure tangents at every point,  but it is not $(\mathcal{L}^{2n+1}, k)$-rectifiable, for any $k\in \mathbb{N}$. The latter follows by the fact that $\mathbb{H}^n$ is purely $k$-unrectifiable for any $k>n$; this was proved for $n=1$ in  \cite{AmbKirch-MathAnn} and in  \cite{Magnani, BaFa-09} for general $n\in \mathbb{N}$.  Notably, the non-collapsing assumption {\slshape ii)} in Theorem \ref{thm:rectifiability_MCP} is enough to rule out such examples and prove rectifiability. 
\end{remark}

\footnotesize{
\bibliography{biblio}

\bibliographystyle{alpha} 
}
\end{document}